\documentclass[hidelinks,onefignum,onetabnum]{siamart220329}

\usepackage{latexsym}
\usepackage{amssymb}
\usepackage{amsmath}
\usepackage{amsfonts}
\usepackage{dsfont}
\usepackage{graphicx}
\usepackage{epstopdf}
\usepackage{algorithmic}
\ifpdf
  \DeclareGraphicsExtensions{.eps,.pdf,.png,.jpg}
\else
  \DeclareGraphicsExtensions{.eps}
\fi

\newsiamremark{remark}{Remark}
\newsiamremark{hypothesis}{Hypothesis}
\crefname{hypothesis}{Hypothesis}{Hypotheses}
\newsiamthm{claim}{Claim}

\headers{Shape Optimization by Least Mean Approximation}{Gerhard Starke}

\title{Shape Optimization by Constrained First-Order System Least Mean Approximation\thanks{Submitted to the editors DATE.}}

\author{Gerhard Starke\thanks{Fakult\"at f\"ur Mathematik, Universit\"at Duisburg-Essen, 45117 Essen, Germany (\email{gerhard.starke@uni-due.de}).}}

\usepackage{amsopn}

\newcommand{\tr}{{\rm tr}}
\newcommand{\id}{{\rm id}}
\renewcommand{\div}{{\rm div}}

\newcommand{\R}{\mathds{R}}

\newcommand{\Z}{\mathds{Z}}

\newcommand{\cA}{{\cal A}}
\newcommand{\cS}{{\cal S}}
\newcommand{\cT}{{\cal T}}

\begin{document}

\maketitle

\begin{abstract}
  In this work, the problem of shape optimization, subject to PDE constraints, is reformulated as an $L^p$ best approximation problem under divergence constraints
  to the shape tensor introduced by Laurain and Sturm in \cite{LauStu:16}. More precisely, the main result of this paper states that the $L^p$ distance
  of the above approximation problem is equal to the dual norm of the shape derivative considered as a functional on $W^{1,p^\ast}$ (where $1/p + 1/p^\ast = 1$).
  This implies that for any given shape, one can evaluate its distance from being a stationary one with respect to the shape derivative by simply solving the
  associated $L^p$-type least mean approximation problem. Moreover, the Lagrange multiplier for the divergence constraint turns out to be the shape deformation of
  steepest descent. This provides a way, as an alternative to the approach by Deckelnick, Herbert and Hinze \cite{DecHerHin:22}, for computing shape gradients in
  $W^{1,p^\ast}$ for $p^\ast \in ( 2 , \infty )$. The discretization of the least mean approximation problem is done with (lowest-order) matrix-valued Raviart-Thomas finite
  element spaces leading to piecewise constant approximations of the shape deformation acting as Lagrange multiplier. Admissible deformations in $W^{1,p^\ast}$ to be
  used in a shape gradient iteration are reconstructed locally. Our computational results confirm that the $L^p$ distance of the best approximation does indeed measure
  the distance of the considered shape to optimality. Also confirmed by our computational tests are the observations from \cite{DecHerHin:22} that choosing $p^\ast$
  (much) larger than 2 (which means that $p$ must be close to 1 in our best approximation problem) decreases the chance of encountering mesh degeneracy during the
  shape gradient iteration.
\end{abstract}



\section{Introduction}
\label{sec-introduction}

The treatment of shape optimization problems have been of concern in the numerical analysis community for some decades already and has recently received
much renewed interest. The state of the art in the field of shape and topology optimization was recently summarized in the excellent survey paper by Allaire,
Dapogny and Jouve \cite{AllDapJou:21}. Much recent research was inspired by the observation that the representation of the shape derivative as a volume
integral not only requires less regularity of the considered shapes but is also advantageous from the view of numerical approximation in finite element spaces
\cite{Ber:10,HipPagSar:15, LauStu:16}. Among the essential contributions based on the volume expression of the shape derivative are also: \cite{SchSieWel:16}, where
a suitable Steklov-Poincar{\'e} metric is developed, \cite{EigStu:18}, where {\color{black} reproducing kernel} Hilbert spaces are used for its evaluation,
\cite{AlbLauYou:21}
for application to time-domain full waveform inversion, and \cite{LauWinYou:21} for application to high-temperature superconductivity. The study of shape optimization
problems under convexity constraints in \cite{BarWac:20} also relies on the volume representation of the shape derivative. Finally, two recent contributions address
the crucial issue of avoiding the degeneration of the meshes during an iterative procedure based on a gradient representation of the shape derivative. In
\cite{EtlHerLoaWac:20} this is achieved by a suitable restriction of the admissible shape deformations to those meaningful by the Hadamard structure theorem.
A different route is taken in \cite{DecHerHin:22} and \cite{DecHerHin:23a}, where the Hilbert space setting is left and shape gradients are instead computed in
$W^{1,p^\ast}$ for $p^\ast > 2$ (ideally for $p^\ast \rightarrow \infty$, i.e., in the Lipschitz topology).

In this paper, the problem of finding an optimal shape with respect to some shape functional and subject to PDE constraints is reformulated as an $L^p$ best
approximation problem in a space of matrix-valued functions with
divergence constraints. The function to be approximated coincides with the shape expression in tensorial form introduced by Laurain and Sturm \cite{LauStu:16}
and studied further by Laurain et.al. in \cite{Lau:20} and \cite{LauLopNak:23}. The main result (Theorem 3.3) states that, for $1 < p < \infty$, the $L^p$ distance
to the Sturm-Laurain shape tensor described above is equal to
to the dual norm of the shape derivative, considered as a functional on $W^{1,p^\ast}$ (with $1/p + 1/p^\ast = 1$) and that the Lagrange multiplier associated
with the divergence constraint is the direction of steepest descent among the shape deformations. The first and immediate implication is that for any given shape,
one can evaluate its distance from optimality (or, more precisely, from having a stationary shape derivative) by simply solving the corresponding $L^p$
best approximation problem. The second implication is that this provides an alternative way to compute shape gradients in $W^{1,p^\ast}$ for $p^\ast > 2$ which is
different from the approach developed in \cite{DecHerHin:22}, \cite{DecHerHin:23a} {\color{black} and \cite{DecHerHin:23b}. These contributions all build upon
the earlier work \cite{MueKueSieDecHinRun:21} utilizing the $p$-Laplace operator in a fluid mechanical context, see also \cite{MuePinRunSie:23}.}

The above-mentioned $L^p$-type least mean approximation problem can be viewed as a generalization of a constrained first-order system least squares formulation,
investigated, for example, in \cite{AdlVas:14} (see also \cite{CaiSta:04}). Therefore, our approach is named ``constrained first-order system least mean'', also
reminding us of the fact that our ultimate goal is to allow $p \rightarrow 1$ (and, thus, $p^\ast \rightarrow \infty$). The finite element approximation of the shape tensor
is done in matrix-valued Raviart-Thomas spaces.
We restrict ourselves to lowest-order since higher-order approximations would also require a higher-order parametrization
of the boundaries and associated parametric Raviart-Thomas spaces \cite{BerSta:16}. The shape deformation arising as Lagrange multiplier for the divergence
constraint is naturally approximated by piecewise constants in order to satisfy the discrete inf-sup condition. From this piecewise constant approximation an admissible
deformation in $W^{1,p^\ast}$ needs to be reconstructed which is possible to do in a local manner without much effort by the procedures in \cite{Ste:91} or
\cite{ErnVoh:15}. 

The structure of this paper is as follows: In the following section we present the classical shape optimization problem constrained by the Poisson equation
with Dirichlet conditions and reformulate it as a least mean approximation problem subject to divergence constraints. Section \ref{sec-structure} investigates
the structure of the necessary conditions for the constrained best approximation problem and presents the main result about the equivalence to the dual norm
of the shape derivative and the steepest descent property of the Lagrange multiplier. The discretization of the best approximation problem in suitable finite
element spaces is the subject of Section \ref{sec-discretization}. Section \ref{sec-shape_gradient_iteration} describes the reconstruction of an admissible
deformation from the piecewise constant approximation of the Lagrange multiplier and uses this to develop an iterative scheme for the computation of the
optimal shape. A simple test example with a known optimal shape in form of a disk is used to illustrate the performance of the shape gradient iteration. This
example is already introduced in the previous sections in order to illustrate the theoretical results. A more sophisticated second example with a non-convex
(and unknown) optimal shape is also used for numerical computations. {\color{black} A third example from the literature with a non-convex optimal shape
where the barycenter is not known before-hand from symmetry considerations complements our numerical experiments.}
Since the functional, given by the $L^p$ best approximation and measuring the distance to stationarity, contains several error components associated with the
geometry and finite element approximations, a detailed analysis is finally carried out in Section \ref{sec-error_analysis}.

\section{Shape optimality as a problem of $L^p$ best approximation}
\label{sec-best_approximation}

We are concerned with PDE-constrained shape optimization among an appropriate set of admissible shapes $\Omega \subset \R^d$ ($d = 2$ or $3$)
associated with minimizing a shape functional
\begin{equation}
  J (\Omega) = \int_\Omega j (u_\Omega) \: dx \: ,
  \label{eq:shape_functional}
\end{equation}
where $u_\Omega \in H_0^1 (\Omega)$ is the solution of the Dirichlet problem
\begin{equation}
  ( \nabla u_\Omega , \nabla v )_{L^2 (\Omega)} = ( f , v )_{L^2 (\Omega)} \mbox{ for all } v \in H_0^1 (\Omega) \: .
  \label{eq:bvp}
\end{equation}
Gradients of scalar functions like $\nabla u_\Omega \in \R^d$ are defined as column vectors. However, for gradients of vector functions like
$\nabla \theta \in \R^{d \times d}$, each row is the gradient of a component of $\theta \in \R^d$. 
From now on, $( \cdot , \cdot )$ always denotes the duality pairing $( L^p (\Omega) , L^{p^\ast} (\Omega) )$ with $1/p + 1/p^\ast = 1$ ($1 < p < \infty$).
For Lipschitz domains and under suitable assumptions on $f$ and $j ( \: \cdot \: )$ (see below), the shape derivative
{\color{black} (in direction $\chi \in W^{1,\infty} (\Omega;\R^d)$)} exists and is given by
\begin{equation}
  \begin{split}
    J^\prime (\Omega) [ \chi ]
    & = \left( \left( (\div \: \chi) \: I - \left( \nabla \chi + (\nabla \chi)^T \right) \right) \nabla u_\Omega , \nabla y_\Omega \right) \\
    & \hspace{4cm} + \left( f \: \nabla y_\Omega , \chi \right) + \left( j (u_\Omega) , \div \: \chi \right) \: ,
  \end{split}
  \label{eq:shape_derivative_pre}
\end{equation}
where $y_\Omega \in H_0^1 (\Omega)$ denotes the solution of the adjoint problem
\begin{equation}
  ( \nabla y_\Omega , \nabla z ) = - ( j^\prime (u_\Omega) , z ) \mbox{ for all } z \in H_0^1 (\Omega)
  \label{eq:adjoint_bvp}
\end{equation}
(cf. \cite[Prop. 4.5]{AllDapJou:21}). Note that this is well-defined {\color{black} for $j : H_0^1 (\Omega) \rightarrow L^1 (\Omega)$, $f \in L^2 (\Omega)$ and}
if $j$ is differentiable with {\color{black} $j^\prime (u) \in L^2 (\Omega)$ for all $u \in H_0^1 (\Omega)$}. We reformulate this for our purposes as
\begin{equation}
  J^\prime (\Omega) [ \chi ] = \left( K ( u_\Omega ,  y_\Omega ) , \nabla \chi \right)
  + \left( f \: \nabla y_\Omega , \chi \right) + \left( j (u_\Omega) , \div \: \chi \right)
  \label{eq:shape_derivative}
\end{equation}
with
\begin{equation}
  K ( u_\Omega , y_\Omega ) = \left( \nabla u_\Omega \cdot \nabla y_\Omega \right) \: I
  - \nabla y_\Omega \otimes \nabla u_\Omega - \nabla u_\Omega \otimes \nabla y_\Omega \: ,
  \label{eq:definition_K}
\end{equation}
where $\otimes$ stands for the outer product. This is exactly the shape expression in tensorial form studied by Laurain and Sturm in \cite{LauStu:16}.
Therefore, we call it Sturm-Laurain shape tensor.
Since, in general, $\nabla u_\Omega$ and $\nabla y_\Omega$ can only be guaranteed to be in $L^2 ( \Omega ; \R^{d \times d} )$, only
$K ( u_\Omega , y_\Omega ) \in L^1 (\Omega ; \R^{d \times d})$ will hold. This implies that, without additional regularity assumptions, $J^\prime (\Omega) [\chi]$
is only defined for $\chi \in W^{1,\infty} (\Omega;\R^d)$.

For the admissible domains we consider the set
\begin{equation}
  \cS = \{ \Omega = (\id + \theta) \Omega_0 : \theta , (\id + \theta)^{-1} - \id \in W^{1,\infty} (\Omega;\R^d) \} \: ,
  \label{eq:set_of_shapes}
\end{equation}
where $\id$ denotes the identity mapping in $\R^d$ and $\Omega_0$ is a fixed reference domain, e.g. the unit disk in $\R^2$ or unit ball in $\R^3$.
In view of (\ref{eq:shape_derivative}) we are looking for a domain $\Omega \in \cS$ such that
\begin{equation}
  ( K ( u_\Omega , y_\Omega ) , \nabla \chi ) + ( j (u_\Omega) , \div \: \chi ) + ( f \: \nabla y_\Omega , \chi ) = 0
  \mbox{ for all } \chi \in W^{1,\infty} (\Omega ; \R^d)
  \label{eq:shape_variational_problem}
\end{equation}
is satisfied for the Sturm-Laurain shape tensor. This is equivalent to finding $\Omega \in \cS$ such that
\begin{equation}
  \begin{array}{rl}
    ( \div \: K ( u_\Omega , y_\Omega ) - f \: \nabla y_\Omega , \chi ) - ( j (u_\Omega) , \div \: \chi ) & = 0 \: , \\
    \langle K ( u_\Omega , y_\Omega ) \cdot n , \chi \rangle & = 0
    \label{eq:shape_optimality_system}
  \end{array}
\end{equation}
holds for all $\chi \in W^{1,\infty} (\Omega ; \R^d)$, where, from now on, $\langle \cdot , \cdot \rangle$ denotes the corresponding duality pairing on the boundary
$\partial \Omega$.

Furthermore, (\ref{eq:shape_optimality_system}) may be reformulated as determining $\Omega \in \cS$ such that, for some $p \in ( 1 , \infty )$ (only $p \in ( 1, 2 ]$ is
of practical interest), there exists an
\[
  S \in \Sigma^p := \{ T \in L^p (\Omega ; \R^{d \times d}) : \div \: T \in L^p (\Omega ; \R^d) \}
\]
with
\begin{equation}
  \begin{array}{rl}
    ( \div \: S , \chi ) & = ( f \: \nabla y_\Omega , \chi ) + ( j (u_\Omega) , \div \: \chi ) \mbox{ for all } \chi \in W^{1,\infty} (\Omega ; \R^d) \: , \\
    \langle S \cdot n , \left. \chi \right|_{\partial \Omega} \rangle & = 0 \hspace{3.95cm} \mbox{ for all } \chi \in W^{1,\infty} (\Omega ; \R^d) \: , \\
    S - K ( u_\Omega , y_\Omega ) & = 0 \: .
  \end{array}
  \label{eq:shape_optimality_system_first_order}
\end{equation}
Motivated by (\ref{eq:shape_optimality_system_first_order}), the use of
\begin{equation}
  \begin{split}
    \eta_p (\Omega) = \inf \{ & \| T - K ( u_\Omega , y_\Omega ) \|_{L^p (\Omega)} : T \in \Sigma^p \mbox{ satisfying } \\
    & \hspace{0.1cm} ( \div \: T , \chi ) = ( f \: \nabla y_\Omega , \chi ) + ( j (u_\Omega) , \div \: \chi )
    \:\: \forall \: \chi \in W^{1,p^\ast} (\Omega ; \R^d) \: , \\
    & \hspace{0.1cm} \langle T \cdot n , \left. \chi \right|_{\partial \Omega} \rangle = 0 \:\: \forall \: \chi \in W^{1,p^\ast} (\Omega;\R^d) \}
  \end{split}
  \label{eq:l_p_functional}
\end{equation}
is proposed, with $1/p + 1/p^\ast = 1$, as an estimator for shape optimality (or, more precisely, shape stationarity) of $\Omega$. In other words, the
$L^p$ distance of the Sturm-Laurain shape tensor $K ( u_\Omega , y_\Omega )$ to the admissible subset of $\Sigma^{p,0}$ measures how far the shape
$\Omega$ is from being a stationary one. In order to make sure that ({\ref{eq:l_p_functional}) is well-defined, the linear functional
\begin{equation}
  F (\chi) := ( f \: \nabla y_\Omega , \chi ) + ( j (u_\Omega) , \div \: \chi )
  \label{eq:constraint_functional}
\end{equation}
needs to be bounded on $W^{1,p^\ast} (\Omega;\R^d)$. This is certainly guaranteed if we assume
$f \in L^\infty (\Omega)$ and $j (u_\Omega) \in L^2 (\Omega)$ which we will tacitly do from now on. For brevity of notation, let us define
\begin{equation}
  \Sigma^{p,0} = \{ T \in \Sigma^p : \langle T \cdot n , \left. \chi \right|_{\partial \Omega} \rangle = 0 \mbox{ for all } \chi \in W^{1,p^\ast} (\Omega ; \R^d) \} \: .
\end{equation}

\begin{theorem}
  Let $p \in ( 1 , 2 ]$ and assume that $K ( u_\Omega , y_\Omega ) \in L^p ( \Omega ; \R^{d \times d} )$. If
  the compatibility condition $( f \: \nabla y_\Omega , e ) = 0$ for all constant $e \in \R^d$ is fulfilled, then
  there is a uniquely determined $S \in \Sigma^{p,0}$ satisfying the constraint
  \begin{equation}
      ( \div \: S , \chi ) = ( f \: \nabla y_\Omega , \chi ) + ( j (u_\Omega) , \div \: \chi ) \mbox{ for all } \chi \in W^{1,p^\ast} ( \Omega;\R^d ) \: ,
    \label{eq:constraints_l_p_minimization}
  \end{equation}
  such that the infimum in (\ref{eq:l_p_functional}) is attained, i.e.,
  \begin{equation}
    \eta_p (\Omega) = \| S - K ( u_\Omega , y_\Omega ) \|_{L^p (\Omega)} \: .
    \label{eq:unique_S}
  \end{equation}
  \label{theorem-unique_S}
\end{theorem}

\begin{proof}
  We start by showing that the admissible set defined by the constraints (\ref{eq:constraints_l_p_minimization}) is not empty. To this end, write
  $\bar{S} = \nabla \psi$ where $\psi \in H^1 (\Omega ; \R^d)$ solves the variational problem
  \begin{equation}
      ( \nabla \psi , \nabla \chi ) = - ( f \: \nabla y_\Omega , \chi ) - ( j (u_\Omega) , \div \: \chi ) \mbox{ for all } \chi \in H^1 (\Omega;\R^d) \: .
    \label{eq:auxiliary_Neumann_problem}
  \end{equation}
  Due to the compatibility condition, (\ref{eq:auxiliary_Neumann_problem}) has a solution which is unique up to constant vectors in $\R^d$.
  From integration by parts, we obtain
  \begin{equation}
      \langle \bar{S} \cdot n , \left. \chi \right|_{\partial \Omega} \rangle - ( \div \: \bar{S} , \chi )
      = - ( f \: \nabla y_\Omega , \chi ) - ( j (u_\Omega) , \div \: \chi ) \mbox{ for all } \chi \in H^1 (\Omega;\R^d) \: .
  \end{equation}
  This proves that $\bar{S} \in H (\div,\Omega;\R^{d \times d})$ is also contained in $\Sigma^{p,0}$  and satisfies the constraint in (\ref{eq:constraints_l_p_minimization})
  and is therefore admissible.

  Since $\| T - K ( u_\Omega , y_\Omega ) \|_{L^p (\Omega)}$ is convex with respect to $T$, (\ref{eq:l_p_functional}) defines a convex minimization problem
  {\color{black} in the non-empty closed subspace defined by the constraints (\ref{eq:constraints_l_p_minimization})}  
  and therefore possesses at least one solution.

  Uniqueness follows from the fact that the space $L^p ( \Omega ; \R^{d \times d} )$ is strictly normed for $1 < p < \infty$ (see e.g. \cite[Theorem 3.3.21]{AtkHan:09}).
\end{proof}

\begin{remark}
  The compatibility condition
  \begin{equation}
    ( f \: \nabla y_\Omega , e ) = 0 \mbox{ for constant } e \in \R^d
    \label{eq:compatibility_condition}
  \end{equation}
  means that, due to (\ref{eq:shape_derivative}), $J^\prime (\Omega) [ e ] = 0$ for all translations $e \in \R^d$. The geometric interpretation is
  that moving its center of gravity does not lead to a reduction of the shape functional. The condition (\ref{eq:compatibility_condition}) will be satisfied in all
  the examples in this paper since the origin will always be the center of gravity of the optimal domain. In the general case, if the compatibility condition is not
  satisfied automatically, our approach needs to be modified by adding an additional step which computes the constant displacement deformations.
\end{remark}

\begin{remark}
  If we make the stronger assumption that $j (u_\Omega) \in H^1 (\Omega)$, then the constraints in (\ref{eq:l_p_functional}) can be rewritten as
  \begin{equation}
    \begin{split}
      ( \div \: S , \chi ) & = ( f \: \nabla y_\Omega , \chi ) - ( \nabla j (u_\Omega) , \chi ) \mbox{ for all } \chi \in W^{1,p^\ast} ( \Omega;\R^d ) \: , \\
      \langle S \cdot n , \left. \chi \right|_{\partial \Omega} \rangle & = \langle j (u_\Omega) , \left. \chi \cdot n \right|_{\partial \Omega} \rangle
      \mbox{ for all } \chi \in W^{1,p^\ast} ( \Omega;\R^d ) \: .
    \end{split}
    \label{eq:constraints_l_p_minimization_modified}
  \end{equation}
  These constraints may be evaluated for more general, possibly discontinuous, test functions $\chi$ which will be important for the finite element
  discretization in Section \ref{sec-discretization}.
\end{remark}

{\em Example 1 (Circular optimal shape).}
\label{example-circle}
We illustrate the behavior of the computed values for $\eta_{p,h} (\Omega_h)$ for a very simple test problem where the optimal shape is a disk.
To this end, set $f = 1/2 - {\bf 1}_D$, where ${\bf 1}_D$ denotes the characteristic function of the unit disk $D \subset \R^2$ and
$j (u_\Omega) \equiv u_\Omega/2$ (with $j^\prime (u_\Omega) \equiv 1/2$). The solution of the boundary value problem (\ref{eq:bvp}) on
$D_R = \{ x \in \R^2 : | x | < R \}$ for $R > 1$ is given by
\begin{equation}
  u_R (x) = \left\{ \begin{array}{ll}
    \frac{R^2 + | x |^2}{8} - \frac{1}{4} - \frac{1}{2} \: \ln R & \: , \: 0 \leq | x | < 1 \: , \\
    \frac{R^2 - | x |^2}{8} + \frac{1}{2} ( \ln | x | - \ln R ) & \: , \: 1 < | x | < R \: ,
  \end{array} \right.
\end{equation}
while the solution of (\ref{eq:adjoint_bvp}) is simply $y_R (x) = (| x |^2 - R^2)/8$. This implies
\begin{equation}
  \nabla u_R (x) = \left\{ \begin{array}{ll}
    \frac{1}{4} \: x  & \: , \: 0 \leq | x | < 1 \: , \\
    - \frac{1}{4} \: x + \frac{1}{2} \: \frac{x}{| x |^2} & \: , \: 1 < | x | < R \: ,
  \end{array} \right. \mbox{ and }
  \nabla y_R (x) = \frac{1}{4} x
\end{equation}
leading to 
\begin{equation}
  \begin{split}
    K ( u_R , y_R ) & = \left( \nabla u_R \cdot \nabla y_R \right) I - \nabla y_R \otimes \nabla u_R - \nabla u_R \otimes \nabla y_R \\
    & = \left\{ \begin{array}{ll}
      \frac{1}{16} | \id |^2 I - \frac{1}{8} \id \otimes \id & \: , \: 0 \leq | x | < 1 \: , \\
      \left( \frac{1}{8} - \frac{1}{16} | \id |^2 \right) I - \left( \frac{1}{4 \: | \id |^2} - \frac{1}{8} \right) \id \otimes \id & \: , \: 1 < | x | < R
    \end{array} \right.
  \end{split}
\end{equation}
as explicit expression for the Sturm-Laurain shape tensor. This implies
\begin{equation}
  \begin{split}
    \div \: K ( u_R , y_R ) = & \left\{ \begin{array}{ll}
      - \frac{1}{4} \id & \: , \: 0 \leq | x | < 1 \: , \\[0.05cm]
      \frac{1}{4} \id - \frac{1}{4 \: | \id |^2} \: \id & \: , \: 1 < | x | < R 
    \end{array} \right\} = f \: \nabla y_R - \nabla j (u_R) \: , \\
    K ( u_R , y_R ) \cdot n & = \left( \frac{1}{8} - \frac{R^2}{16} \right) n - \left( \frac{1}{4} - \frac{R^2}{8} \right) n \mbox{ on } \partial D_R \: ,
  \end{split}
  \label{eq:constraints_K}
\end{equation}
which means that (\ref{eq:shape_optimality_system}) is satisfied for $R = \sqrt{2}$. The direct calculation of the shape functional for disks with radius $R \geq 1$
using polar coordinates gives
\begin{equation}
  \begin{split}
    J (D_R) = \frac{1}{2} \int_{D_R} u_R \: dx
    & = \pi \int_0^1 \left( \frac{R^2 + r^2}{8} - \frac{1}{4} - \frac{1}{2} \: \ln R \right) \: r \: dr \\
    & \hspace{1.5cm} + \pi \int_1^R \left( \frac{R^2 - r^2}{8} + \frac{1}{2} (\ln r - \ln R) \right) \: r \: dr \\
    & = \frac{\pi}{32} \left( R^4 - 4 R^2 + 2 \right) = \frac{\pi}{32} \left( R^2 - 2 \right)^2 - \frac{\pi}{16} \: ,
  \end{split}
\end{equation}
which confirms that the minimum value is $J (D_{\sqrt{2}}) = - \pi / 16$.
  
For disks $D_R$, it is possible to calculate $\eta_p (D_R)$ by hand as follows. Recall that our problem consists in finding the $L^p (D_R)$ best approximation
$S \in \Sigma^p$ subject to
\begin{equation}
  \begin{split}
    \div \: S & = f \: \nabla u_R - \nabla j (u_R) = \left\{ \begin{array}{ll}
      - \frac{1}{4} \id & \: , \: 0 \leq | x | < 1 \: , \\[0.05cm]
      \frac{1}{4} \id - \frac{1}{4 \: | \id |^2} \: \id & \: , \: 1 < | x | < R 
    \end{array} \right. \: , \\
    S \cdot n & = j (u_R) \: n = 0 \: .
  \end{split}
  \label{eq:constraints_S}
\end{equation}
Combining (\ref{eq:constraints_K}) and (\ref{eq:constraints_S}) leads to
\[
  \div \: ( S - K ( u_R , y_R ) ) = 0 \mbox{ and } ( S - K ( u_R , y_R ) ) \cdot n = \left( \frac{1}{8} - \frac{R^2}{16} \right) \: n \: .
\]
These two constraints are obviously fulfilled by setting
\begin{equation}
  S - K ( u_R , y_R ) = \left( \frac{1}{8} - \frac{R^2}{16} \right) \: I
  \label{eq:ex1_radius_optimal_choice}
\end{equation}
and we show that (\ref{eq:ex1_radius_optimal_choice}) is the optimal choice. To this end, we calculate the Fr\'{e}chet derivative
\begin{equation}
  \begin{split}
    D_S \left( \| S - K ( u_R , y_R ) \|_{L^p}^p \right) [ T ] & = p \: ( | S - K ( u_R , y_R ) |^{p-2} ( S - K ( u_R , y_R ) ) , T ) \\
    & = p \: d^{p/2-1} \: \left| \frac{1}{8} - \frac{R^2}{16} \right|^{p-2} \: \left( \frac{1}{8} - \frac{R^2}{16} \right) \: ( I , T ) \: .
  \end{split}
\end{equation}
From $( I , T ) = ( \nabla \id , T ) = \langle \id , T \cdot n \rangle - ( \id , \div \: T )$
we get$D_S \left( \| S - K ( u_R , y_R ) \|_{L^p}^p \right) [ T ] = 0$ for all $T \in \Sigma^p$ with $\div \: T = 0$ in $D_R$ and $T \cdot n = 0$ on $\partial D_R$
proving optimality. Therefore, we get
\begin{equation}
  \eta_p (D_R) = \left| \frac{1}{8} - \frac{R^2}{16} \right| \: \| I \|_{L^p (D_R;\R^{2 \times 2})} = \frac{\left| 2 - R^2 \right|}{16} \sqrt{2} \left( \pi R^2 \right)^{1/p} \: .
  \label{eq:eta_disk_R}
\end{equation}
For more general domains $\Omega$, we will compute finite element approximations $\eta_{p,h} (\Omega_h)$ for $\eta_p (\Omega)$ in Section
\ref{sec-discretization}.

\section{The structure of the constrained least mean approximation problem}
\label{sec-structure}

The minimization problem associated with the determination of (\ref{eq:l_p_functional}) is a convex minimization problem in $\Sigma^{p,0}$. For $1 < p < \infty$ it is
differentiable and the associated KKT conditions for $S \in \Sigma^{p,0}$ being a minimizer are
\begin{equation}
  \begin{split}
    ( | S - K (u_\Omega,y_\Omega) |^{p-2} (S - K (u_\Omega,y_\Omega)) , T ) + ( \div \: T , \theta ) & = 0 \: , \\
    ( \div \: S , \chi ) - ( f \: \nabla y_\Omega , \chi ) - ( j (u_\Omega) , \div \: \chi ) & = 0
  \end{split}
  \label{eq:KKT_2}
\end{equation}
for all $T \in \Sigma^{p,0}$ and $\chi \in W^{1,p^\ast} (\Omega;\R^d)$ with a Lagrange multiplier $\theta \in W^{1,p^\ast} (\Omega;\R^d)$.

{\color{black}
\begin{remark}
  The proof of the next theorem relies on the validity of a Helmholtz-type decomposition $L^{p^\ast} (\Omega;\R^d) = G^{p^\ast} (\Omega) \oplus H_{p^\ast} (\Omega)$
  into a subspace of gradients
  \begin{equation}
    G_{p^\ast} (\Omega) = \{ \nabla \vartheta : \vartheta \in W^{1,p^\ast} (\Omega) \}
  \end{equation}
  and a subspace $H_{p^\ast} (\Omega)$ of divergence-free functions with zero trace. Such a decomposition is known to exist for all $p^\ast \in ( 1 , \infty )$
  for domains $\Omega$ with $C^2$ boundary (cf. \cite{FujMor:77}, \cite[Sect. III.1]{Gal:11}). It is also known that this is not true for general Lipschitz domains
  $\Omega$, in which case it is proved to be true for $3/2 \leq p^\ast \leq 3$ (see \cite[Sect. 11]{FabMenMit:98}). The existence of such a Helmholtz-type
  decomposition is related to the solvability of a specific Neumann boundary value problem for the Poisson equation which, for domains with corners, depends
  on the size of the interior angles. In our context, this Helmholtz decomposition is applied individually for each row of a matrix-valued function in
  $L^p (\Omega;\R^{d \times d})$.
  \label{remark-Helmholtz}
\end{remark}
}

\begin{theorem}
  Let $p \in ( 1 , 2 ]$ {\color{black} be such that $L^{p^{\ast}} (\Omega;\R^{d \times d})$ possesses a Helmholtz decomposition in the sense of Remark
  \ref{remark-Helmholtz}} and assume that $K (u_\Omega,y_\Omega) \in L^p (\Omega;\R^{d \times d})$ as well as
  $f \in L^\infty (\Omega)$ and $j (u_\Omega) \in L^2 (\Omega)$ hold. Moreover, assume that the compatibility condition
  $( f \: \nabla y_\Omega , e ) = 0$ for all constant $e \in \R^d$ is fulfilled.
  Then, (\ref{eq:KKT_2}) has a solution
  \[
    (S,\theta) \in \Sigma^{p,0} \times W^{1,p^\ast} (\Omega;\R^d) \;\;\; ( 1/p + 1/p^\ast = 1 ) \: . 
  \]
  $S \in \Sigma^{p,0}$ is the unique minimizer of (\ref{eq:l_p_functional}).
  $\theta \in W^{1,p^\ast} (\Omega;\R^d)$ is unique up to additive constant vectors. 
  \label{eq:saddle_point_problem_p}
\end{theorem}

\begin{proof}
  From Theorem \ref{theorem-unique_S} we deduce that there is a unique $S \in \Sigma^{p,0}$ which minimizes (\ref{eq:l_p_functional}). Thus, it satisfies
  the necessary condition
  \begin{equation}
    ( | S - K (u_\Omega,y_\Omega) |^{p-2} (S - K (u_\Omega,y_\Omega)) , T ) = 0
    \label{eq:necessary_subspace}
  \end{equation}
  for all $T \in \Sigma^{p,0}_{\rm sol} = \{ Z \in \Sigma^{p,0} : \div \: Z = 0 \}$. Therefore, the Helmholtz decomposition (for each row) in
  $L^{p^\ast} (\Omega;\R^{d \times d})$ (see Remark \ref{remark-Helmholtz}) states that
  \begin{equation}
    | S - K (u_\Omega,y_\Omega) |^{p-2} (S - K (u_\Omega,y_\Omega)) = \nabla \theta
    \label{eq:multiplier_gradient}
  \end{equation}
  with some $\theta \in W^{1,p^\ast} (\Omega;\R^d)$. Integration by parts leads to the first equation in (\ref{eq:KKT_2}).
  Obviously, $\theta \in W^{1,p^\ast} (\Omega;\R^{d \times d})$ is uniquely determined up to adding arbitrary constant vectors.
\end{proof}

\begin{remark}
  The existence of the Lagrange multiplier can also be deduced from the classical theory of constrained optimization (see e.g. \cite[Section 1.3]{ItoKun:08}).
  In order to apply the theory, it needs to be shown that the divergence operator is surjective as a mapping from
  $\Sigma^{p,0}$ to $\left. L^p (\Omega;\R^d) \right|_{\R^d}$. This is closely related to the Helmholtz decomposition used in the above proof
  (see \cite{ArnScoVog:88}).
\end{remark}

For $p \neq 2$, (\ref{eq:KKT_2}) is a nonlinear problem which needs to be solved in an iterative way. Choosing $p \in ( 1 , 2 )$ close to 1
widens the generality of the problems that can be treated.
Note that if the regularity of the underlying boundary value (constraint) problem is such that
$\{ \nabla u_\Omega , \nabla y_\Omega \} \subset H^s (\Omega;\R^d)$ for some $s > 0$, then by the Sobolev embedding theorem
(see, e.g. \cite[Ch. 7]{Leo:23}) $H^s (\Omega;\R^d) \hookrightarrow L^{2q} (\Omega;\R^d)$ with $1 < q < 1 / (1 - 2 s / d)$ and we also have
$K (u_\Omega,y_\Omega) \in L^q (\Omega;\R^{d \times d})$ with these values of $q$.
In the Hilbert space case $p = 2$, the optimality conditions (\ref{eq:KKT_2}) become linear: Find
$S \in \Sigma^{2,0} = H_{\partial \Omega} (\div,\Omega;\R^{d \times d})$ and $\theta \in H^1 (\Omega;\R^d)$ such that
\begin{equation}
  \begin{split}
    ( S , T ) + ( \div \: T , \theta ) & = ( K (u_\Omega,y_\Omega) , T ) \: , \\
    ( \div \: S , \chi ) & = ( f \: \nabla y_\Omega , \chi ) + ( j (u_\Omega) , \div \: \chi )
  \end{split}
  \label{eq:KKT_3}
\end{equation}
holds for all $T \in H_{\partial \Omega} (\div,\Omega;\R^{d \times d})$ and $\chi \in H^1 (\Omega;\R^d)$.


We now return to the general case $p \in ( 1 , 2 ]$ and show that the functional defined by (\ref{eq:l_p_functional}) is related to the dual norm of $J^\prime (\Omega)$.
Moreover, the Lagrange multiplier $\theta \in W^{1,p^\ast} (\Omega;\R^d)$ constitutes the direction of steepest descent.

\begin{theorem}
  Let $p \in ( 1 , 2 ]$ {\color{black} be such that $L^{p^\ast} (\Omega;\R^{d \times d})$ possesses a Helmholtz decomposition in the sense of Remark
  \ref{remark-Helmholtz}} and assume that $K (u_\Omega,y_\Omega) \in L^p (\Omega;\R^{d \times d})$ as well as
  $f \in L^\infty (\Omega)$ and $j (u_\Omega) \in L^2 (\Omega)$ hold and that the compatibility condition
  $( f \: \nabla y_\Omega , e ) = 0$ for all constant $e \in \R^d$ is fulfilled.
  Then, the Lagrange multiplier $\theta \in W^{1,p^\ast} (\Omega;\R^d)$ from (\ref{eq:KKT_2}) satisfies
  \begin{equation}
    - \frac{J^\prime (\Omega) [\theta]}{\| \nabla \theta \|_{L^{p^\ast} (\Omega;\R^{d \times d})}}
    = \sup_{\chi \in W^{1,p^\ast} (\Omega;\R^d)} \frac{J^\prime (\Omega) [\chi]}{\| \nabla \chi \|_{L^{p^\ast} (\Omega;\R^{d \times d})}} = \eta_p (\Omega) \: .
    \label{eq:inf_shape_functional_pstar}
  \end{equation}
  \label{theorem-inf_shape_functional_pstar}
\end{theorem}

\begin{proof}
  The definition of the shape derivative (\ref{eq:shape_derivative}) together with the second equation in (\ref{eq:KKT_2}) and integration by parts yields
  \begin{equation}
    \begin{split}
      J^\prime (\Omega) [\chi] & = ( K (u_\Omega,y_\Omega) , \nabla \chi ) + ( j (u_\Omega) , \div \: \chi ) + ( f \: \nabla y_\Omega , \chi ) \\
      & = ( K (u_\Omega,y_\Omega) , \nabla \chi ) + ( \div \: S , \chi ) \\
      & = ( K (u_\Omega,y_\Omega) - S , \nabla \chi ) + \langle S \cdot n , \chi \rangle = - ( S - K (u_\Omega,y_\Omega) , \nabla \chi )
    \end{split}
    \label{eq:shape_derivative_S}
  \end{equation}
  for all $\chi \in W^{1,p^\ast} (\Omega;\R^d)$. If we now insert $\theta \in W^{1,p^\ast} (\Omega;\R^d)$ satisfying (\ref{eq:multiplier_gradient})
  into (\ref{eq:shape_derivative_S}), then we obtain
  \begin{equation}
    \begin{split}
      J^\prime (\Omega) [\theta] & = - ( S - K (u_\Omega,y_\Omega) , | S - K (u_\Omega,y_\Omega) |^{p-2} (S - K (u_\Omega,y_\Omega)) ) \\
      & = - \| S - K (u_\Omega,y_\Omega) \|_{L^p (\Omega;\R^{d \times d})}^p \: .
    \end{split}
    \label{eq:numerator_optimal}
  \end{equation}
  On the other hand,
  \begin{equation}
    \begin{split}
      \| \nabla \theta & \|_{L^{p^\ast} (\Omega;\R^d)} = \left( \int_\Omega | \nabla \theta |^{p^\ast} dx \right)^{1/p^\ast}
      = \left( \int_\Omega | S - K (u_\Omega,y_\Omega) |^{(p-1) p^\ast} dx \right)^{1/p^\ast} \\
      & = \left( \int_\Omega | S - K (u_\Omega,y_\Omega) |^p dx \right)^{1/p^\ast} 
      = \| S - K (u_\Omega,y_\Omega) \|_{L^p (\Omega;\R^{d \times d})}^{p/p^\ast} \\
      & = \| S - K (u_\Omega,y_\Omega) \|_{L^p (\Omega;\R^{d \times d})}^{p-1} \: .
    \end{split}
  \end{equation}
  Combining (\ref{eq:numerator_optimal}) and (\ref{eq:denominator_optimal}) with (\ref{eq:shape_derivative_S}) leads to
  \begin{equation}
    \begin{split}
      - \frac{J^\prime (\Omega) [\theta]}{\| \nabla \theta \|_{L^{p^\ast} (\Omega;\R^{d \times d})}}
      & = \| S - K (u_\Omega,y_\Omega) \|_{L^p (\Omega;\R^{d \times d})} \\
      & \geq \sup_{\chi \in W^{1,p^\ast} (\Omega;\R^d)} \frac{( S - K (u_\Omega,y_\Omega) , \nabla \chi )}{\| \nabla \chi \|_{L^{p^\ast} (\Omega;\R^{d \times d})}} \\
      & = \sup_{\chi \in W^{1,p^\ast} (\Omega;\R^d)} \frac{J^\prime (\Omega) [\chi]}{\| \nabla \chi \|_{L^{p^\ast} (\Omega;\R^{d \times d})}}
      \geq - \frac{J^\prime (\Omega) [\theta]}{\| \nabla \theta \|_{L^{p^\ast} (\Omega;\R^{d \times d})}} \: .
    \end{split}
    \label{eq:denominator_optimal}
  \end{equation}
  Obviously, equality needs to hold in (\ref{eq:denominator_optimal}) which completes the proof.
\end{proof}

\begin{remark}
  Theorem \ref{theorem-inf_shape_functional_pstar} states that $\theta$ is the steepest descent direction for the shape functional in the $W^{1,p^\ast} (\Omega)$
  {\color{black}semi-}norm {\color{black} which constitutes a norm in
  \begin{equation}
    \Theta^{p^\ast} := \{ \chi \in W^{1,p^\ast} (\Omega;\R^d) : ( \chi , e ) = 0 \mbox{ for all constant } e \in \R^d \} \: .
  \end{equation}
  }
  For $p^\ast \rightarrow \infty$, this represents another way of accessing the shape gradient with respect to the $W^{1,\infty}$ topology
  recently studied by Deckelnick, Herbert and Hinze in \cite{DecHerHin:22} {\color{black} and \cite{DecHerHin:23a} and the $W^{1,p}$ approach
  investigated in \cite{MueKueSieDecHinRun:21} and \cite{DecHerHin:23a}. Due to the approximation of the shape tensors, our approach is expected
  to be more costly than the more direct ones in \cite{MueKueSieDecHinRun:21} and \cite{DecHerHin:23a}. On the other hand, the supply of the
  error measure $\eta_p (\Omega)$ may be worth the additional effort.}
\end{remark}

{\color{black}
\begin{remark}
  Due to the compatibility condition assumed in Theorem \ref{theorem-inf_shape_functional_pstar}, it follows from (\ref{eq:shape_derivative_pre}) that for
  constant $e \in \R^d$,
  \begin{equation}
    J^\prime (\Omega) [ e ] = ( f \: \nabla y_\Omega , e ) = 0
    \label{eq:shape_derivative_translation}
  \end{equation}
  holds. In order to obtain a unique solution of (\ref{eq:KKT_2}) also for the multiplier $\theta$, it also needs to be constrained, e.g. by requiring $( \theta , e ) = 0$
  for all $e \in \R^d$. Due to (\ref{eq:shape_derivative_translation}) this is consistent with the statement of Theorem \ref{theorem-inf_shape_functional_pstar} since
  the constant deformations, i.e., translations, have no effect on (\ref{eq:inf_shape_functional_pstar}).
\end{remark}
}

\section{Discretization of the least mean approximation problem by finite elements}
\label{sec-discretization}

On a triangulation $\cT_h$ of a polyhedrally bounded approximation $\Omega_h$ of $\Omega$, we use lowest-order Raviart-Thomas ($RT_0$) elements
for the approximation of (each row of) $S$ in the finite-dimensional subspace $\Sigma_h^p \subset \Sigma^p$. Combined with the space
$\Theta_h \: {\color{black} \subset \Theta^{p^\ast}}$ of piecewise
constants {\color{black} (normalized to average zero on $\Omega_h$, see Remark 3.5)} for the (component-wise) approximation of the deformation field
$\theta \in W^{1,p^\ast} (\Omega;\R^d)$, this forms an inf-sup stable combination
(cf. \cite{BofBreFor:13}) for the saddle point problem (\ref{eq:KKT_3}). In order to be able to insert discontinuous test functions, we need to resort to the
alternative formulation (\ref{eq:constraints_l_p_minimization_modified}) of the constraints,
\begin{equation}
  \begin{split}
    ( \div \: S , \chi ) & = ( f \: \nabla y_\Omega , \chi ) - ( \nabla j (u_\Omega) , \chi ) \mbox{ for all } \chi \in L^{p^\ast} ( \Omega;\R^d ) \: , \\
    \langle S \cdot n , \chi_b \rangle & = \langle j (u_\Omega) , \chi_b \cdot n \rangle \mbox{ for all } \chi_b \in L^{p^\ast} ( \partial \Omega;\R^d ) \: ,
  \end{split}
  \label{eq:constraints_l_p_minimization_modified_enlarged}
\end{equation}
where we inserted the enlarged test spaces $L^{p^\ast} ( \Omega;\R^d )$ and $L^{p^\ast} ( \partial \Omega;\R^d )$.

With respect to the formulation (\ref{eq:constraints_l_p_minimization_modified_enlarged}) of the constraints, the necessary conditions for the minimization
problem associated with the computation of $\eta_p (\Omega)$ are now given as finding $S \in \Sigma^p$, such that
\begin{equation}
  \begin{split}
    ( | S - K (u_\Omega,y_\Omega) |^{p-2} (S - K (u_\Omega,y_\Omega)) , T ) + ( \div \: T , \theta ) + \langle T \cdot n , \theta_b \rangle & = 0 \: , \\
    ( \div \: S , \chi ) - ( f \: \nabla y_\Omega , \chi ) + ( \nabla j (u_\Omega) , \chi ) & = 0 \: , \\
    \langle S \cdot n , \chi_b \rangle - \langle j (u_\Omega) , \chi_b \cdot n \rangle & = 0
  \end{split}
  \label{eq:KKT_4}
\end{equation}
for all $T \in \Sigma^p$, $\chi \in L^{p^\ast} (\Omega;\R^d)$ and $\chi_b \in L^{p^\ast} (\partial \Omega;\R^d)$. Here, an additional Lagrange multiplier
$\theta_b$ occurs which can be shown to coincide with the boundary values $\left. -\theta \right|_{\partial \Omega}$. The corresponding discretized
problem reads: Find $S_h \in \Sigma_h^p$, $\theta_h \in \Theta_h$ and $\theta_{b,h} \in \Theta_{b,h}$ such that
\begin{equation}
  \begin{split}
    ( | S_h - K (u_{\Omega,h},y_{\Omega,h}) |^{p-2} (S_h - K (u_{\Omega,h},y_{\Omega,h})) , T_h ) + ( \div \: T_h , \theta_h ) \;\;\;\;\;\; & \\
    + \langle T_h \cdot n , \theta_{b,h} \rangle & = 0 \: , \\
    ( \div \: S_h , \chi_h ) - ( f \: \nabla y_{\Omega,h} , \chi_h ) + ( \nabla j (u_{\Omega,h}) , \chi_h ) & = 0 \: , \\
    \langle S_h \cdot n , \chi_{b,h} \rangle - \langle j (u_{\Omega,h}) , \chi_{b,h} \cdot n \rangle & = 0
  \end{split}
  \label{eq:KKT_h}
\end{equation}
for all $T_h \in \Sigma_h^p$, $\chi_h \in \Theta_h$ and $\chi_b \in \Theta_{b,h}$. As discrete space $\Theta_{b,h} \subset L^{p^\ast} (\partial \Omega;\R^d)$
piecewise constant functions on the boundary triangulation are used. $u_{\Omega,h}$ and $y_{\Omega,h}$ are finite element approximations (piecewise linear,
conforming in our computations) of (\ref{eq:bvp}) and (\ref{eq:adjoint_bvp}), respectively. Solving (\ref{eq:KKT_h}) gives approximate values
$\eta_{p,h} (\Omega_h) = \| S_h - K (u_{\Omega,h},y_{\Omega,h}) \|_{L^p (\Omega_h)}$ for $\eta_p (\Omega)$. For $p = 2$, (\ref{eq:KKT_h}) constitutes
a linear saddle point problem while an iterative procedure needs to be employed for $p < 2$. For the numerical results presented below, the following rather simple
fixed point scheme was used by us: Determine, in step $k$, $S_h^{(k)} \in \Sigma_h^p$, $\theta_h \in \Theta_h$ and $\theta_{b,h} \in \Theta_{b,h}$ such that
\begin{equation}
  \begin{split}
    ( | S_h^{(k-1)} - K (u_{\Omega,h},y_{\Omega,h}) |^{p-2} (S_h^{(k)} - K (u_{\Omega,h},y_{\Omega,h})) , T_h ) + ( \div \: T_h , \theta_h ) \;\;\; & \\
    + \langle T_h \cdot n , \theta_{b,h} \rangle & = 0 \: , \\
    ( \div \: S_h^{(k)} , \chi_h ) - ( f \: \nabla y_{\Omega,h} , \chi_h ) + ( \nabla j (u_{\Omega,h}) , \chi_h ) & = 0 \: , \\
    \langle S_h^{(k)} \cdot n , \chi_{b,h} \rangle - \langle j (u_{\Omega,h}) , \chi_{b,h} \cdot n \rangle & = 0
  \end{split}
  \label{eq:KKT_iterative}
\end{equation}
holds for all $T_h \in \Sigma_h^p$, $\chi_h \in \Theta_h$ and $\chi_b \in \Theta_{b,h}$. For values of $p$ approaching 1, however, this iterative procedure has its
limitations and we will need to investigate more sophisticated solution methods for (\ref{eq:KKT_h}) in the future.

{\em Example 1; continued.}
Now we are able to compute approximations $\eta_{p,h} (\Omega)$ to $\eta_p (\Omega)$ for more general domains $\Omega$ than just the concentric disks
considered in Section \ref{sec-best_approximation}. We investigate the different domains of variable closeness to the optimal shape shown in Figure \ref{fig-shapeex1},
which are all scaled to have the same area $| \Omega | = 2 \pi$ (shared by the optimal shape $D_{\sqrt{2}}$): a square (solid line), an octagon (dashed line)
and a hexadecagon (16 edges, dotted line).

\begin{figure}[h!]
  \hspace{1cm}\includegraphics[scale=0.25]{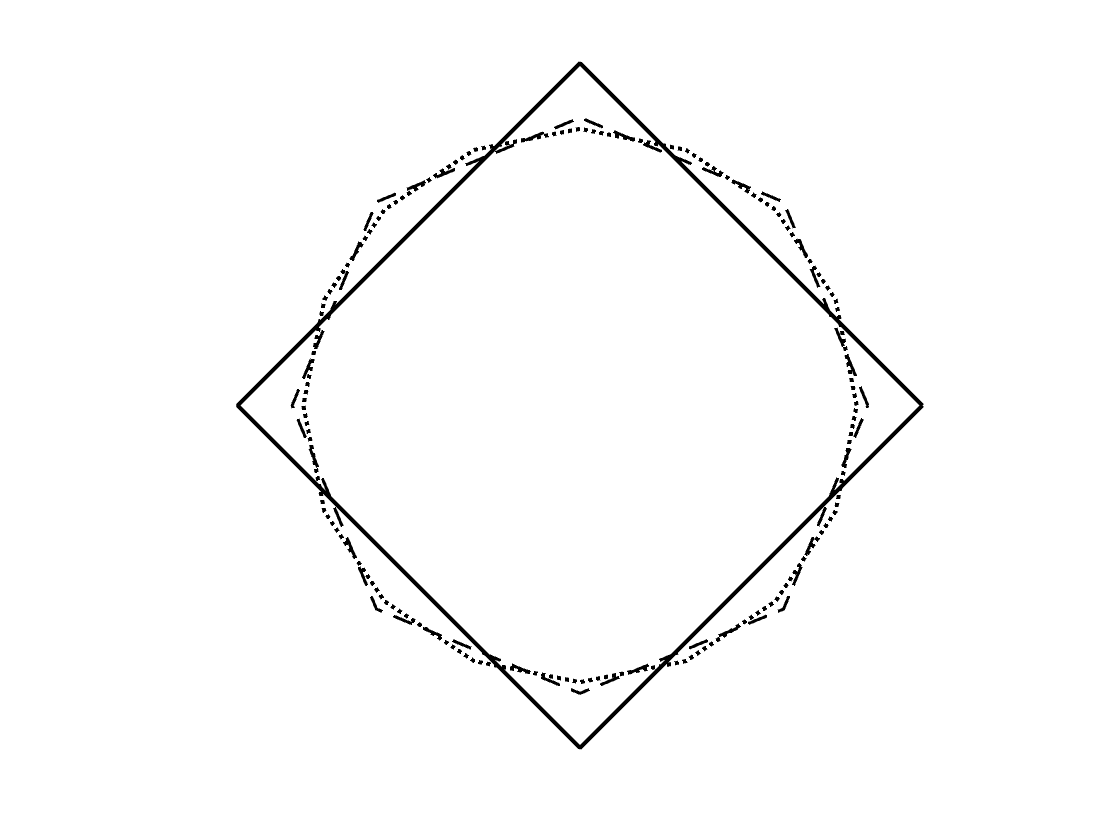}
  \caption{Example 1: Square, octagon, hexadecagon approximating a disk of the same area}
  \label{fig-shapeex1}
\end{figure}

\begin{figure}[h!]
  \hspace{1cm}\includegraphics[scale=0.26]{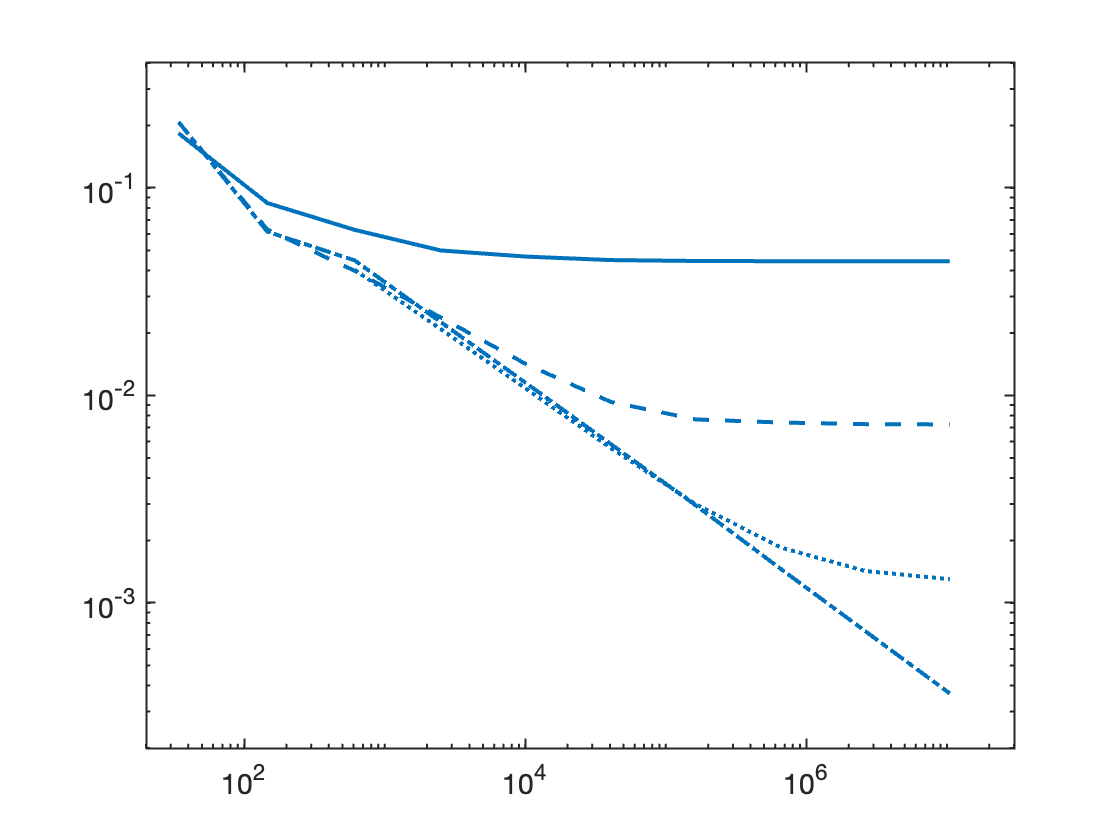}
  \caption{Example 1: Values $\eta_{2,h} (\Omega_h)$ for square, octagon and hexadecagon}
  \label{fig-shapeex1eta2}
\end{figure}

\begin{figure}[h!]
  \hspace{1cm} \includegraphics[scale=0.26]{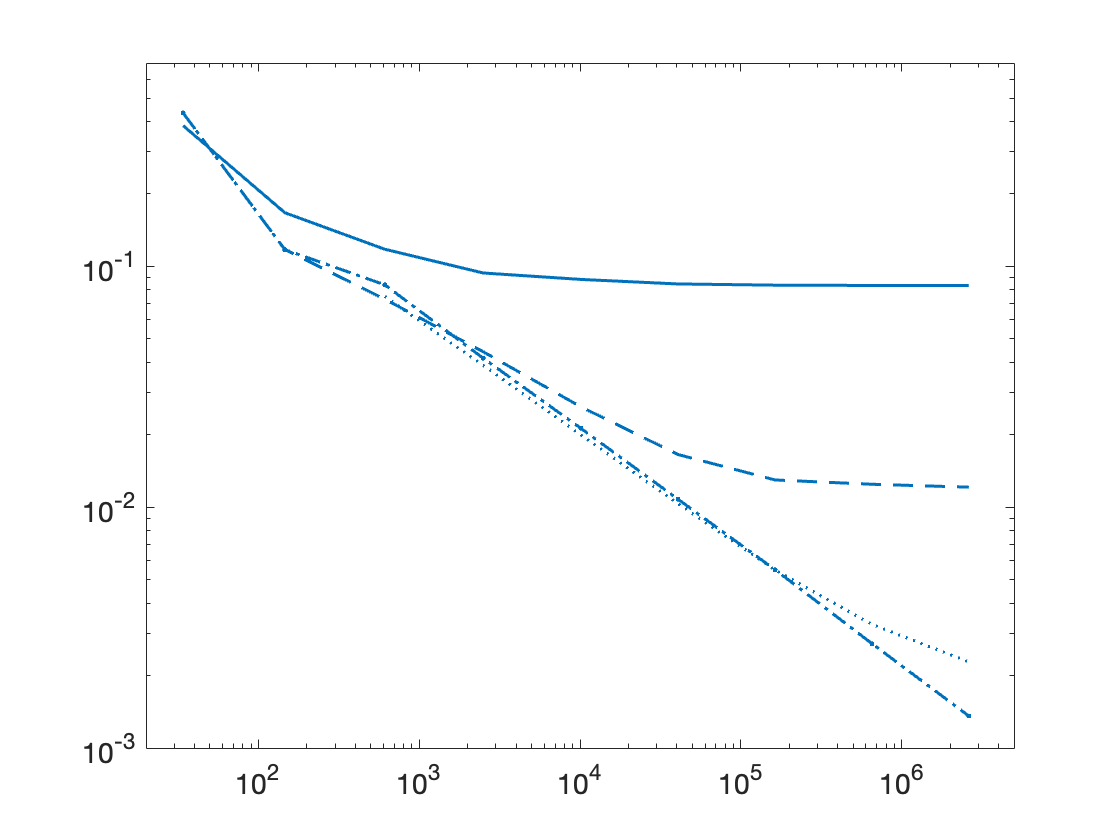}
  \caption{Example 1: Values $\eta_{1.1,h} (\Omega_h)$ for square, octagon and hexadecagon}
  \label{fig-shapeex1eta1point1}
\end{figure}

Figures \ref{fig-shapeex1eta2} and \ref{fig-shapeex1eta1point1} show the behavior of $\eta_{2,h} (\Omega)$ and $\eta_{1.1,h} (\Omega)$, respectively, (vertical axis)
in dependence of the number of elements (horizontal axis) for $p = 2$ and $1.1$ on a sequence of triangulations resulting from uniform refinement. For the square
(solid lines), the functional is reduced somewhat on refined triangulations before it stagnates at its distance to stationarity. For the octagon (dashed lines), the
reduction goes further until it stagnates at a lower value, meaning that it is closer, in the mathematical strict sense of Theorem
\ref{theorem-inf_shape_functional_pstar}, to being a stationary shape. For the hexadecagon (dotted lines), the behavior is repeated with an even smaller
asymptotical value. Finally, the dash-dotted lines show the behavior of $\eta_{2,h} (\Omega_h)$ and $\eta_{1.1,h}$ for polygonal approximations to the optimal disk
$D_{\sqrt{2}}$ tending to zero at a rate proportionally to $h$ (the inverse of the square root of the number of elements). In this case, $\eta_{p,h} (\Omega_h)$ also
contains the geometry error associated with the approximation $\Omega_h$ of $D_{\sqrt{2}}$. For the polygonal domains above, $\Omega_h = \Omega$ holds and
the difference between $\eta_{p,h} (\Omega)$ and $\eta_p (\Omega)$ is due only to the finite element approximations of (\ref{eq:KKT_h}) and of the underlying
boundary value problems (\ref{eq:bvp}) and (\ref{eq:adjoint_bvp}).

\begin{figure}[h!]

  \vspace{-0.35cm}
  
  \hspace*{-0.1cm}\includegraphics[scale=0.33]{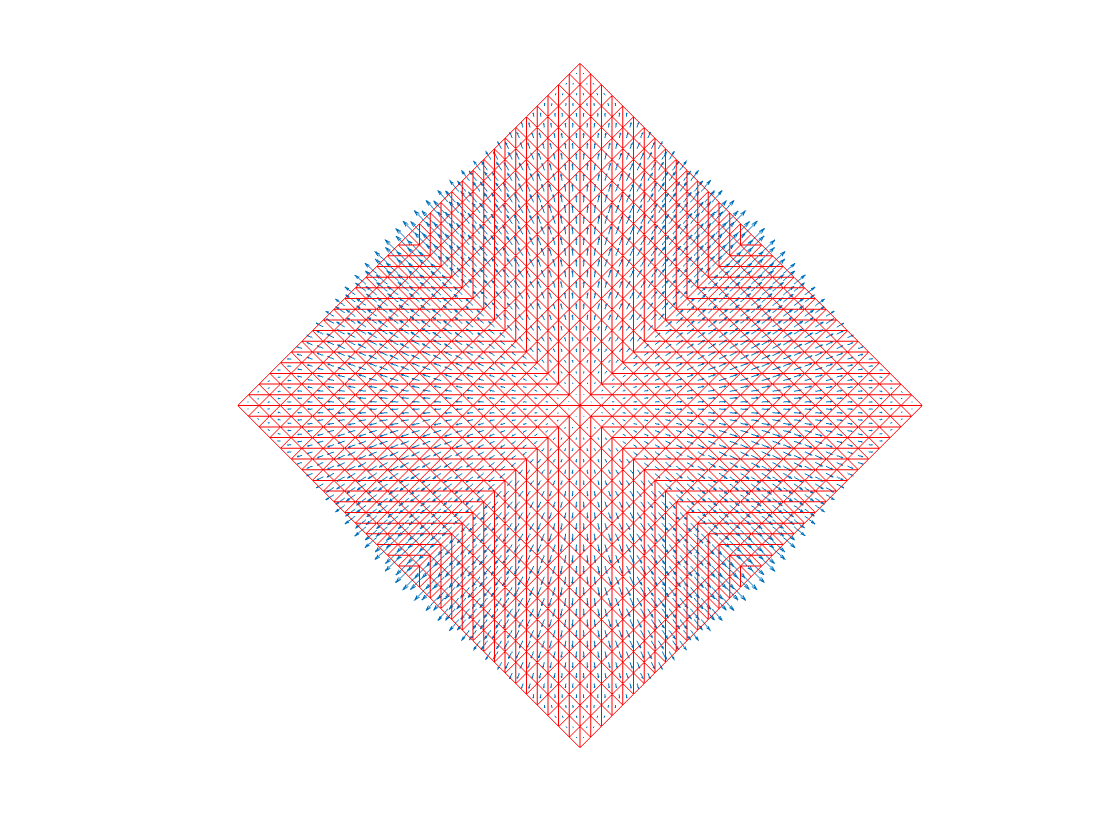}
  
  \vspace{-0.5cm}
  
  \caption{Example 1: Piecewise constant deformation $\theta_h$ for $p = 1.1$}
  \label{fig-shapeex1defo}
\end{figure}
  
The piecewise constant deformation field $\theta_h$ which comes as a by-product of solving (\ref{eq:KKT_h}) is shown in Figure
\ref{fig-shapeex1defo} for the square. The triangulation shown is the result of 4 uniform refinements of an initial one with 8 triangles (refinement level 4).
The graph shows that $\theta_h$ points into the right direction in order to deform $\Omega$
towards the optimal disk. But due to its discontinuous nature it is not useful for the construction of a deformed domain. We will use this
piecewise constant deformation field for the reconstruction of a piecewise linear continuous deformation field as outlined below.

\section{Shape gradient iteration}

\label{sec-shape_gradient_iteration}

The finite element approximation $\theta_h$ obtained from (\ref{eq:KKT_h}) for the steepest descent deformations is not continuous across element
borders. It is therefore not usable for computing a deformed shape $(\id + \theta_h) (\Omega_h)$. Instead, a continuous deformation $\theta_h^\diamond$ is
required in order to lead to an admissible domain $(\id + \theta_h^\diamond) (\Omega_h)$ with a non-degenerate triangulation $(\id + \theta_h^\diamond) (\cT_h)$.
The first idea that comes to mind would be to determine $\theta_h^\diamond \in \left. W^{1,p^\ast} (\Omega_h;\R^d) \right|_{\R^d}$ as a piecewise linear
(continuous) function such that
\begin{equation}
  \| \nabla \theta_h^\diamond - \left| S_h - K ( u_{\Omega,h} , y_{\Omega,h} ) \right|^{p-2} ( S_h - K ( u_{\Omega,h} , y_{\Omega,h} ) ) \|_{L^{p^\ast} (\Omega_h)}
  \label{eq:potential_reconstruction_global}
\end{equation}
is minimized. {\color{black} This would give us a good approximation $\nabla \theta_h^\diamond$ to $\nabla \theta$, depending on the accuracy of $S_h$
approximating $S$. By Poincar{\'e}'s inequality, due to the normalization of $\theta_h^\diamond$ to mean value zero, this would also approximate $\theta$ well.
However, (\ref{eq:potential_reconstruction_global})}
is a global variational problem which we would like to avoid because of its computational cost. Instead, we compute the continuous
deformation $\theta_h^\diamond \in W^{1,p^\ast} (\Omega_h;\R^d)$ using the following local potential reconstruction procedure:
\begin{algorithm}
  1. Compute, for each element $\tau \in \cT_h$, $\left. \nabla \theta_h^\Box \right|_\tau \in \R^{d \times d}$ such that
  \[
    \| \nabla \theta_h^\Box - \left| S_h - K ( u_{\Omega,h} , y_{\Omega,h} ) \right|^{p-2} ( S_h - K ( u_{\Omega,h} , y_{\Omega,h} ) ) \|_{L^{p^\ast} (\tau)}
    \longrightarrow \min!
  \]
  2. Compute, for each element $\tau \in \cT_h$, $\theta_h^\Box$ with $\left. \nabla \theta_h^\Box \right|_\tau$ given by 1. such that
  \[
    \| \theta_h^\Box - \theta_h \|_{L^{p^\ast} (\tau)} \longrightarrow \min!
  \]
  3. Compute $\theta_h^\diamond$ piecewise linear, continuous, by averaging at the vertices:
  \[
    \theta_h^\diamond (\nu) = \frac{1}{| \{ \tau \in \cT_h : \nu \in \tau \} |} \sum_{\tau \in \cT_h : \nu \in \tau} \theta_h^\Box (\left. \nu \right|_\tau) \: .
  \]
\end{algorithm}

Obviously, the above algorithm involves only local computations. In the first step, a nonlinear system of equations of dimension $d^2$ needs to be solved on each
element $\tau \in \cT_h$. The second step involves a nonlinear system of equations of dimension $d$ and the third step consists in averaging a small number of
terms. For $p^\ast = 2$, the algorithm reduces to (the lowest-order case of) the classical reconstruction technique by
Stenberg (see \cite{Ste:91}). An alternative would be to compute $\theta_h^\diamond$ by minimizing (\ref{eq:potential_reconstruction_global}) on a decomposition
into vertex patches as in \cite[Sect. 3]{ErnVoh:15}. {\color{black} The key idea is that one can locally combine the knowledge of a good approximation for
$\nabla \theta$ (given by $| S_h - K (u_{\Omega,h},y_{\Omega,h}) |^{p-2} (S_h - K (u_{\Omega,h},y_{\Omega,h}))$ from the solution of (\ref{eq:KKT_h})) with a good
approximation for the value of $\theta$ itself (given by $\theta_h$ from the solution of (\ref{eq:KKT_h})).}

With $\theta_h^\diamond$, an updated shape
\begin{equation}
  \Omega_h^\diamond = (\id + \alpha \theta_h^\diamond) (\Omega_h)
  \label{eq:updated_shape}
\end{equation}
can be computed, where $\alpha \in \R$ is a suitable step-size chosen in order to guarantee a certain reduction of the shape functional $J (\Omega_h^\diamond)$,
e.g. according to an Armijo-type rule. In our computations presented in the sequel, we put a little more effort into the one-dimensional minimization and approximately
solved it by
\begin{equation}
  \alpha = 2^{k^\ast} \: , \: k^\ast = \arg \min_{k \in \Z} J ((\id + 2^k \theta_h^\diamond) (\Omega_h))
  \label{eq:stepsize_discrete}
\end{equation}
i.e., with respect to a logarithmic discretization. It is reasonable to spend this additional effort since the evaluation of the shape functional is rather cheap in
comparison with the computation of the shape gradient, especially for $p < 2$.

{\em Example 1; continued.}
We test the above shape gradient iteration by trying to recover the optimal shape $D_{\sqrt{2}}$ starting with the tetragon {\color{black} with corners at
$(1,0)$, $(0,1)$, $(-1,0)$ and $(0,-1)$} as initial guess.
{\color{black} Figure \ref{fig-example_1_shapiter_level4} shows the final shape $\Omega_h^\diamond$ for $p = 2$ (left half) and $p = 1.1$ (right half) and the
corresponding triangulation on refinement level 4 (2048 triangles).} The final shape for $p = 1.1$ looks closer, in the ``eyeball norm'' to the optimal circle (dotted line)
than the one for $p = 2$, which also manifests itself in the smaller value{\color{black}s} of the shape functional {\color{black} in Table
\ref{table-example1_1point1} compared to those in Table \ref{table-example1_2}.
The functional values $\eta_{2,h} (\Omega_h^\diamond)$ are reduced less regularly with decreasing $h$ than the corresponding values
$\eta_{1.1,h} (\Omega_h^\diamond)$. The values themselves are hard to compare with each other for different $p$ since they are based on different norms.
Naturally, all the  values lie above $\eta_{2,h} (\Omega_h)$ or $\eta_{1.1,h} (\Omega_h)$, respectively, for the polygonal approximation $\Omega_h$ of the
optimal disk (see Figures \ref{fig-shapeex1eta1point1} and \ref{fig-shapeex1eta2}).

\begin{figure}[h!]

  \vspace{-0.5cm}

  \hspace{-0.65cm}\includegraphics[scale=0.34]{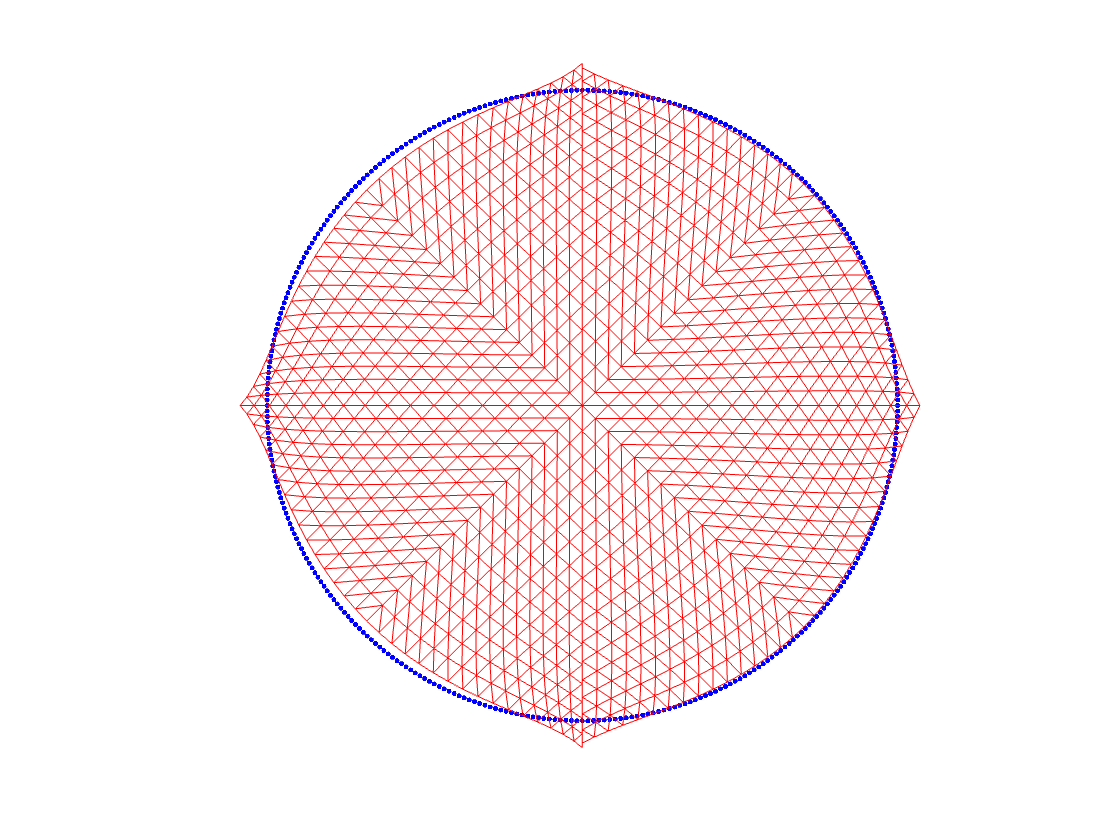}
  
  \vspace{-0.6cm}
  
  \caption{\color{black} Shape gradient iteration for $p = 2$ (left half) and $p = 1.1$ (right half)}
  \label{fig-example_1_shapiter_level4}
\end{figure}

Figure \ref{fig-example_1_shapiter_level4} also shows the potential of shape gradients in $W^{1,p^\ast} (\Omega)$ with $p^\ast > 2$ for straightening corners
as observed already in the numerical experiments in \cite{DecHerHin:23a} and utilized to produce good approximations to ``desired'' sharp corners for fluid
mechanical shape optimization in \cite{MueKueSieDecHinRun:21} and \cite{MuePinRunSie:23}. The corner angle of $\pi/2$ of the initial shape (tetragon) is
only slightly straightened by the shape gradient iteration for $p = 2$ less so on refined meshes. The results associated with $p = 1.1$ come significantly closer
to the value $\pi$ of the optimal shape (disk). It should also be noted that the boundary value problems (\ref{eq:bvp}) and (\ref{eq:adjoint_bvp}) are approximated
on the same mesh which is part of the reason for the relatively large discrepancy from the optimal disk in Figure \ref{fig-example_1_shapiter_level4}.

\begin{table}[h!]
  \centering
  {\small \color{black}
  \begin{tabular}{|c|cccc|}
    \hline $| \cT_h |$ & 2048 & 8192 & 32768 & 131072 \\ \hline
    $J (\Omega_h^\diamond)$ & $-0.19074$ & $-0.19459$ & $-0.19559$ & $-0.19596$ \\
    $\eta_{2,h} (\Omega_h^\diamond)$ & \;\;\;\;\;$2.4863 \cdot 10^{-2}$ & \;\;\;\;\;$1.1577 \cdot 10^{-2}$ & \;\;\;\;\;$8.4113 \cdot 10^{-3}$ & \;\;\;\;\;$5.2711 \cdot 10^{-3}$ \\
    corner angle & $\;\;\;0.5704 \pi$ & $\;\;\;0.5321 \pi$ & $\;\;\;0.5140 \pi$ & $\;\;\;0.5035 \pi$ \\ \hline
  \end{tabular}
  }
  \caption{\color{black} Shape approximation for $p = 2$}
  \label{table-example1_2}
\end{table}

\begin{table}[h!]
  \centering
  {\small \color{black}
  \begin{tabular}{|c|cccc|}
    \hline $| \cT_h |$ & 2048 & 8192 & 32768 & 131072 \\ \hline
    $J (\Omega_h^\diamond)$ & $-0.19546$ & $-0.19588$ & $-0.19597$ & $-0.19629$ \\
    $\eta_{1.1,h} (\Omega_h^\diamond)$ & \;\;\;\;\;$3.9901 \cdot 10^{-2}$ & \;\;\;\;\;$1.1799 \cdot 10^{-2}$ & \;\;\;\;\;$7.5568 \cdot 10^{-3}$ & \;\;\;\;\;$3.0363 \cdot 10^{-3}$ \\
    corner angle & $\;\;\;0.6768 \pi$ & $\;\;\;0.6163 \pi$ & $\;\;\;0.6154 \pi$ & $\;\;\;0.5540 \pi$ \\ \hline
  \end{tabular}
  }
  \caption{\color{black} Shape approximation for $p = 1.1$}
  \label{table-example1_1point1}
\end{table}
}

{\em Example 2 (``gingerbread man'').}
This example is taken from \cite{BarWac:20} and uses 
\begin{align*}
  f (x) & = - \frac{1}{2} + \frac{4}{5} \: | x |^2 + 2 \sum_{i=1}^5 \exp \left( - 8 \: | x - y^{(i)} |^2 \right)
  - \sum_{i=1}^5 \exp \left( - 8 \: | x - z^{(i)} |^2 \right) \\
  \mbox{ with } & \;\;\;\; y^{(i)} = \left( \sin ( \frac{(2 i + 1) \pi}{5} ) , \cos ( \frac{(2 i + 1) \pi}{5} ) \right) \: , \: i = 1 , \ldots , 5 \: , \\
  & \;\;\;\; z^{(i)} = \left( \frac{6}{5} \sin ( \frac{2 i \pi}{5} ) , \frac{6}{5} \cos ( \frac{2 i \pi}{5} ) \right) \: , \: i = 1 , \ldots , 5
\end{align*}
(and $j (u_\Omega) = u_\Omega / 2$ as in Example 1). The optimal shape is not convex and was used to study the effect of convexity constraints in
\cite{BarWac:20}. For us, it serves the purpose of a more challenging shape optimization problem as a test for our method. Starting from the unit disk
as the initial shape, we perform shape gradient iteration steps until convergence or until the mesh degenerates. On refinement level 4 (2048 triangles),
the final results are shown in Figure \ref{fig-example_2_shapiter_2_level4} for $p = 2$ and in Figure \ref{fig-example_2_shapiter_1point1_level4} for $p = 1.1$.

\begin{figure}[h!]

  \vspace{-0.25cm}

  \includegraphics[scale=0.28]{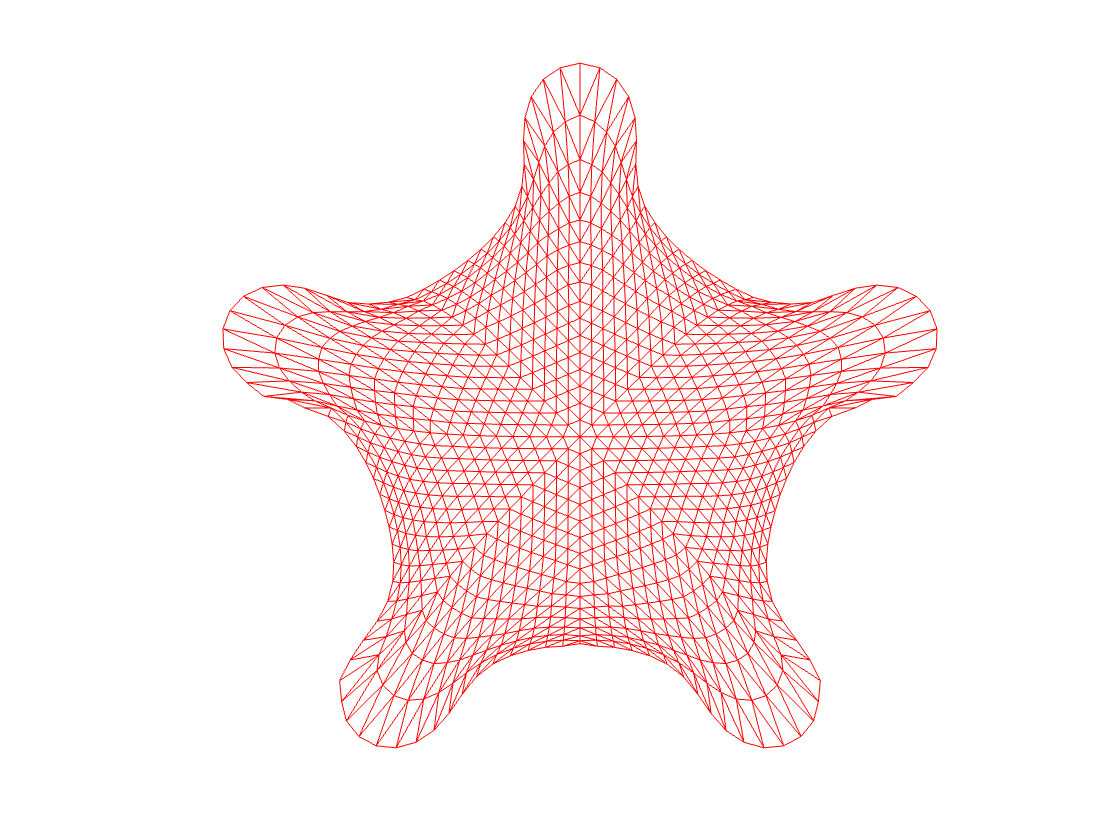}
  
  \vspace{-0.5cm}
  
  \caption{Shape approximation for $p = 2$: shape iteration terminates due to degenerate mesh}
  \label{fig-example_2_shapiter_2_level4}
\end{figure}

\begin{figure}[h!]

  \vspace{-0.25cm}

  \includegraphics[scale=0.28]{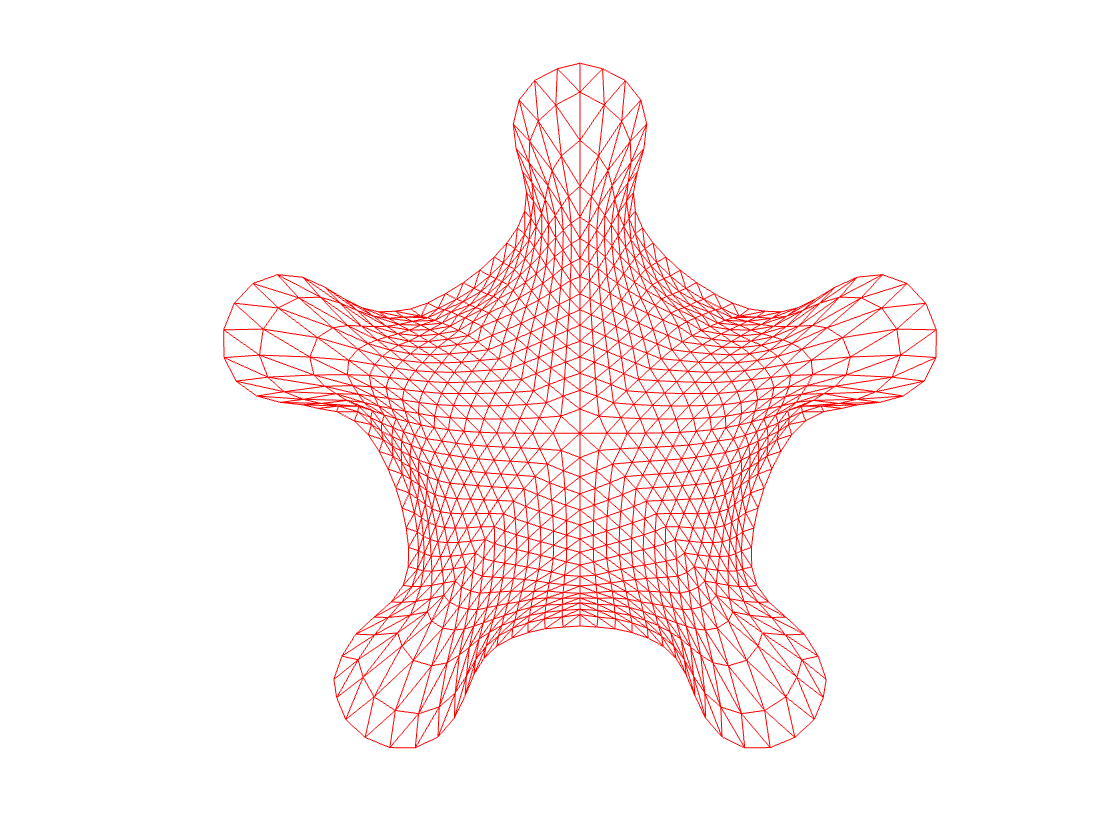}
  
  \vspace{-0.5cm}
  
  \caption{Shape approximation for $p = 1.1$: shape iteration converges}
  \label{fig-example_2_shapiter_1point1_level4}
\end{figure}

The triangulation shown in Figure \ref{fig-example_2_shapiter_2_level4}, obtained by the shape gradient iteration with $p = 2$, is severely degenerate which is the
reason for early termination. The most distorted triangles are located at the underside of the left and right ``arm'' of the gingerbread man. In contrast, the graph in
Figure \ref{fig-example_2_shapiter_1point1_level4}, associated with $p = 1.1$, is the result of a converged solution and shows finer details of the optimal shape which
develop during the final iterations. This confirms the observation made in \cite{DecHerHin:22} that shape gradients in $W^{1,p^\ast}$ with $p^\ast > 2$ (in our case
$p^\ast = 1/(1 - 1/1.1) = 11$) improve the quality of the meshes encountered during the iteration.

\begin{table}[h!]
  \centering
  {\small
  \begin{tabular}{|c|cccc|}
    \hline $| \cT_h |$ & 2048 & 8192 & 32768 & 131072 \\ \hline
    $J (\Omega_h^\diamond)$ & $-1.41655 \cdot 10^{-2}$ & $-1.45463 \cdot 10^{-2}$ & $-1.48595 \cdot 10^{-2}$ & $-1.49520 \cdot 10^{-2}$ \\
    $\eta_{2,h} (\Omega_h^\diamond)$ & \;\;$1.7669 \cdot 10^{-3}$ & \;\;$2.3609 \cdot 10^{-3}$ & \;\;$1.1601 \cdot 10^{-3}$ & \;\;$2.9467 \cdot 10^{-4}$ \\ \hline
  \end{tabular}
  }
  \caption{Shape approximation for $p = 2$: shape iteration terminates due to degenerate mesh}
  \label{table-example_2_shapiter_2_level4}
\end{table}

\begin{table}[h!]
  \centering
  {\small
  \begin{tabular}{|c|cccc|}
    \hline $| \cT_h |$ & 2048 & 8192 & 32768 & 131072 \\ \hline
    $J (\Omega_h^\diamond)$ & $-1.42676 \cdot 10^{-2}$ & $-1.47522 \cdot 10^{-2}$ & $-1.49038 \cdot 10^{-2}$ & $-1.49517 \cdot 10^{-2}$ \\
    $\eta_{1.1,h} (\Omega_h^\diamond)$ \!\!\! & \;\;$2.9972 \cdot 10^{-3}$ &\; \;$1.4073 \cdot 10^{-3}$ & \;\;$7.6197 \cdot 10^{-4}$ & \;\;$3.9894 \cdot 10^{-4}$ \\ \hline
  \end{tabular}
  }
  \caption{Shape approximation for $p = 1.1$: shape iteration converges}
  \label{table-example_2_shapiter_1point1_level4}
\end{table}

The superiority of using $p < 2$ in our shape gradient iteration is illustrated quantitatively by the numbers in Tables \ref{table-example_2_shapiter_2_level4}
and \ref{table-example_2_shapiter_1point1_level4}. The minimum value of the shape functional $J (\Omega_h^\diamond)$ reached for $p = 1.1$ is smaller than
that for $p = 2$. The gap between these minimal values is more pronounced on coarser meshes and closes by refinement. The distance to
the optimal shape, measured by the stationarity functional $\eta_{p,h} (\Omega_h^\diamond)$, also decreases as the triangulation is refined. Remarkably,
the behavior seems to be proportional to $h$ asymptotically which suggests that $\Omega_h$ tends to the optimal shape at this rate for $h \rightarrow 0$.
For $p=2$, the iteration terminated prematurely because of mesh degeneracy on the two coarser triangulations while convergence was achieved on the two
finer ones.

{\color{black}

\begin{figure}[h!]

  \vspace{-0.45cm}

  \hspace{-1.15cm}\includegraphics[scale=0.245]{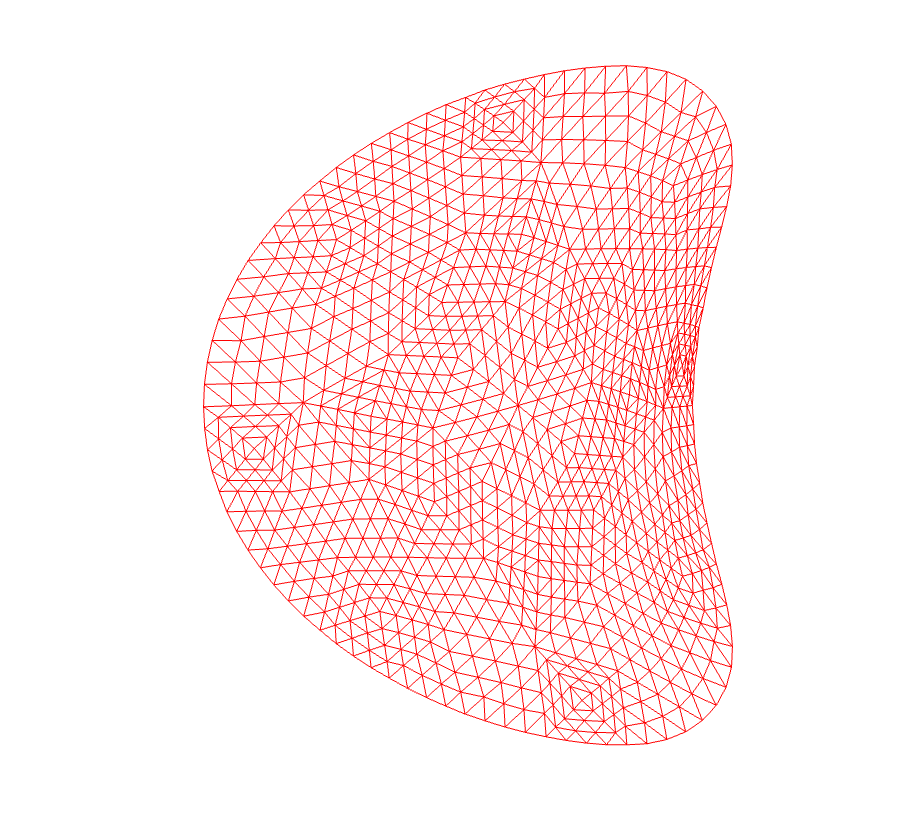}\hspace{-1.45cm}\includegraphics[scale=0.245]{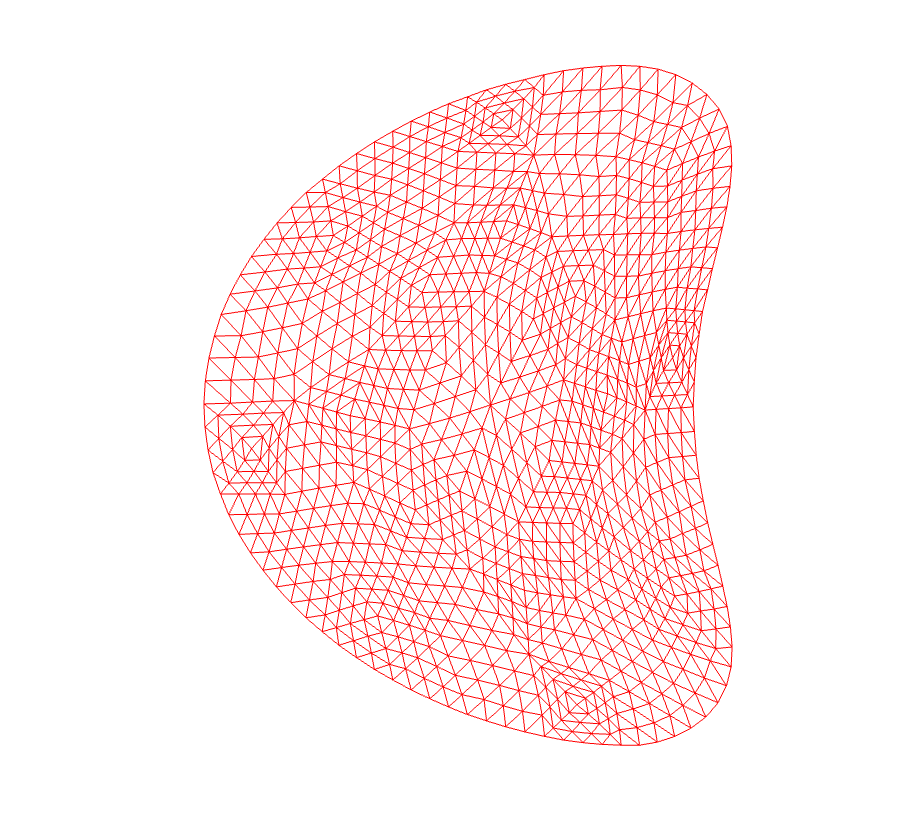}
  
  \vspace{-0.5cm}
  
  \caption{\color{black} Shape approximation for $p = 2$ (left) and $p = 1.1$ (right)}
  \label{fig-example_3_shapiter_level4}
\end{figure}

{\em Example 3 (``kidney shape'').} This example is taken from \cite{EtlHerLoaWac:20} (see also \cite{HerLoa:24}) and uses
\[
  f (x) = \frac{5}{2} \left( x_1 + \frac{2}{5} - x_2^2 \right)^2 + x_1^2 + x_2^2 - 1
\]
(and again $j (u_\Omega) = u_\Omega / 2$ as in our previous examples). The optimal shape (which is not explicitly known) is again not convex. In contrast
to our previous examples, the barycenter is also not known beforehand which means that the iteration (\ref{eq:updated_shape}) based on $\theta_h \in \Theta_h$
having zero mean value must be complemented with planar translations. This is realized in our computations by the following two-step shape gradient iteration:
\begin{equation}
  \begin{split}
    \Omega_h^\odot & = (\Omega_h + e^\odot) \: , \\
    \Omega_h^\diamond & = (\id + \alpha \theta_h^\diamond) (\Omega_h^\odot) \: ,
  \end{split}
  \label{eq:updated_shape_two_step}
\end{equation}
where $e^\odot \in \R^d$ in the first half-step is determined in such a way that
\begin{equation}
  J^\prime (\Omega_h + e^\odot) = ( f \: \nabla y_{\Omega_h + e^\odot} , e ) = 0 \mbox{ for all } e \in \R^d \: .
\end{equation}
The step-size $\alpha$ in the second half-step is again selected according to (\ref{eq:stepsize_discrete}).

\begin{table}[h!]
  \centering
  {\small \color{black}
  \begin{tabular}{|c|cccc|}
    \hline $| \cT_h |$ & 2048 & 8192 & 32768 & 131072 \\ \hline
    $J (\Omega_h^\diamond)$ & $-4.63464\cdot 10^{-2}$ & $-4.66937 \cdot 10^{-2}$ & $-4.68261 \cdot 10^{-2}$ & $-4.68700 \cdot 10^{-2}$ \\
    $\eta_{2,h} (\Omega_h^\diamond)$ & \;\;\;$4.3334 \cdot 10^{-3}$ & \;\;\;$2.1687 \cdot 10^{-3}$ & \;\;\;$1.0897 \cdot 10^{-3}$ & \;\;\;$5.4797 \cdot 10^{-4}$ \\
    barycenter & $-3.1088 \cdot 10^{-1}$ & $-3.0932 \cdot 10^{-1}$ & $-3.0885 \cdot 10^{-1}$ & $-3.0926 \cdot 10^{-1}$ \\ \hline
  \end{tabular}
  }
  \caption{\color{black} Shape approximation for $p = 2$}
  \label{table-example_3_shapiter_2_level4}
\end{table}

\begin{table}[h!]
  \centering
  {\small \color{black}
  \begin{tabular}{|c|cccc|}
    \hline $| \cT_h |$ & 2048 & 8192 & 32768 & 131072 \\ \hline
    $J (\Omega_h^\diamond)$ & $-4.63393 \cdot 10^{-2}$ & $-4.67017 \cdot 10^{-2}$ & $-4.68305 \cdot 10^{-2}$ & $-4.68711 \cdot 10^{-2}$ \\
    $\eta_{1.1,h} (\Omega_h^\diamond)$ \!\!\! & \;\;\;$5.5282 \cdot 10^{-3}$ &\;\;\;$2.7656 \cdot 10^{-3}$ & \;\;\;$1.3867 \cdot 10^{-3}$ & \;\;\;$6.9501 \cdot 10^{-4}$ \\
    barycenter & $-3.1049 \cdot 10^{-1}$ & $-3.0949 \cdot 10^{-1}$ & $-3.0914 \cdot 10^{-1}$ & $-3.0925 \cdot 10^{-1}$ \\ \hline
  \end{tabular}
  }
  \caption{\color{black} Shape approximation for $p = 1.1$}
  \label{table-example_3_shapiter_1point1_level4}
\end{table}

Figure \ref{fig-example_3_shapiter_level4} shows the final shapes and associated triangulations. The mesh quality in the vicinity of the non-convex part of
the shape is again improved by going from $p = 2$ (left graph) to $p = 1.1$ (right graph). Again, this leads to smaller values of the shape functional
for $p = 1.1$ (Table \ref{table-example_3_shapiter_1point1_level4}) compared to $p = 2$ (Table \ref{table-example_3_shapiter_2_level4}) although the
differences are less pronounced as in the previous example. Note the factor $1/2$ in the scaling of the objective function with respect to the results in \cite{HerLoa:24}.
The decreasing behavior of $\eta_{p,h} (\Omega_h^\diamond)$ proportionally to $h$ is again clearly visible in Table \ref{table-example_3_shapiter_2_level4} and Table
\ref{table-example_3_shapiter_1point1_level4} and supports the use of this quantity as a measure for the closeness of the approximate shape to being optimal.
}

\section{Error Analysis}
\label{sec-error_analysis}

In the computational results of the previous sections, numerical values for $\eta_{p,h} (\Omega_h)$ were presented and interpreted as a measure for closeness
of the shape $\Omega_h$ to optimality. In contrast to $\eta_p (\Omega_h)$, which by Theorem \ref{theorem-inf_shape_functional_pstar} indeed gives an
indication of how close to stationarity the shape $\Omega_h$ is, $\eta_{p,h} (\Omega_h)$ also contains the discretization errors associated with solving
(\ref{eq:KKT_2}) as well as (\ref{eq:bvp}) and (\ref{eq:adjoint_bvp}). We must therefore show that the difference $\eta_p (\Omega_h) - \eta_{p,h} (\Omega_h)$
is comparably small in order to rule out the possibility that the desired value $\eta_p (\Omega_h)$ is much bigger than the computed one $\eta_{p,h} (\Omega_h)$.

Since $\eta_p (\Omega)$ depends on the exact solutions $u_\Omega$ and $y_\Omega$ of (\ref{eq:bvp}) and (\ref{eq:adjoint_bvp}), while $\eta_{p,h} (\Omega)$
depends on the finite element approximations $u_{\Omega,h}$ and $y_{\Omega,h}$, we need an estimate of the difference between
$K ( u_\Omega , y_\Omega )$ and $K ( u_{\Omega,h} , y_{\Omega,h} )$ to start with.

\begin{lemma}
  For $p \in ( 1, 2 ]$, if $\{ u_\Omega , y_\Omega \} \subset W^{1,q} (\Omega)$ with $q = 2 p /(2-p)$, then
  \begin{equation}
    \begin{split}
      \| K ( u_\Omega , y_\Omega ) & - K ( u_{\Omega,h} , y_{\Omega,h} ) \|_{L^p (\Omega)} \\
      & \leq (2 d)^{1/2} \left( \| \nabla u_\Omega \|_{L^q (\Omega)} \| \nabla (y_\Omega - y_{\Omega,h}) \|_{L^2 (\Omega)} \right. \\
      & \hspace{4cm} \left. + \| \nabla y_\Omega \|_{L^q (\Omega)} \| \nabla (u_\Omega - u_{\Omega,h}) \|_{L^2 (\Omega)} \right)
    \end{split}
    \label{eq:K_bound_difference}
  \end{equation}
  holds.
  \label{lemma-K_bound_difference}
\end{lemma}

\begin{proof}
  (i) The first step consists in showing that, for $u , v , y , z \in H^1 (\Omega)$, we have
  \begin{equation}
    | K ( u , y ) - K ( v , z ) |^2 \leq 2 d \left( | \nabla u |^2 \: | \nabla (y - z) |^2 + | \nabla (u - v) |^2 \: | \nabla z |^2 \right)
    \label{eq:K_bound_difference_pointwise}
  \end{equation}
  (almost everywhere) in $\Omega$. This is the consequence of a purely algebraic argument relying on the structure of $K ( \: \cdot \: , \: \cdot \: )$ as follows:
  Let us introduce the linear mapping $\cA : \R^{d \times d} \rightarrow \R^{d \times d}$ by
  \[
    \cA (M) = \tr (M) \: I - M - M^T
  \]
  and note that the definition of $K ( \: \cdot \: , \: \cdot \: )$ in (\ref{eq:definition_K}) gives $K ( v , z ) = \cA ( \nabla v \otimes \nabla z )$. Therefore,
  \begin{equation}
    \begin{split}
      K ( u , y ) - K ( v , z ) & = \cA ( \nabla u \otimes \nabla y - \nabla v \otimes \nabla z ) \\
      & = \cA ( \nabla u \otimes \nabla (y - z) + \nabla (u - v) \otimes \nabla z \\
      & = \cA ( \nabla u \otimes \nabla (y - z) ) + \cA ( \nabla (u - v) \otimes \nabla z ) \: ,
    \end{split}
  \end{equation}
  which leads to
  \begin{equation}
    | K ( u , y ) - K ( v , z ) |^2 \leq 2 \left( | \cA ( \nabla u \otimes \nabla (y - z) ) |^2 + | \cA ( \nabla (u - v) \otimes \nabla z ) |^2 \right) \: .
    \label{eq:K_bound_difference_pointwise_step1}
  \end{equation}
  For arbitrary $a , b \in \R^d$, one has (for $d \geq 2$)
  \begin{equation}
    \begin{split}
      | \cA ( a \otimes b ) |^2 & = | \tr (a \otimes b) \: I - a \otimes b - b \otimes a |^2 = | (a \cdot b) \: I - a \otimes b - b \otimes a |^2 \\
      & = ( a \cdot b )^2 | I |^2 + | a \otimes b |^2 + | b \otimes a |^2 \\
      & \;\;\;\; - 2 ( a \cdot b ) ( I : (a \otimes b) ) - 2 ( a \cdot b ) ( I : (b \otimes a) ) + 2 ( a \otimes b ) : ( b \otimes a ) \\
      & = d \: ( a \cdot b )^2 + 2 | a |^2 | b |^2 - 4 ( a \cdot b )^2 + 2 ( a \cdot b )^2 \\
      & = (d - 2) \: ( a \cdot b )^2 + 2 | a |^2 | b |^2 \leq d \: | a |^2 | b |^2 \: .
    \end{split}
    \label{eq:rank_1_estimate}
  \end{equation}
  Inserting (\ref{eq:rank_1_estimate}) (with $a = \nabla u , b = \nabla (y - z)$ and with $a = \nabla (u - v) , b = \nabla z$, respectively) into
  (\ref{eq:K_bound_difference_pointwise_step1}) implies (\ref{eq:K_bound_difference_pointwise}).
  
  (ii) From (\ref{eq:K_bound_difference_pointwise}), we obtain
  \begin{align*}
    | K ( u , y ) - K ( v , z ) |^p & \leq (2 d)^{p/2} \left( | \nabla u |^2 | \nabla (y - z) |^2 + | \nabla (u - v) |^2 | \nabla z |^2 \right)^{p/2} \\
    & \leq (2 d)^{p/2} \left( | \nabla u | | \nabla (y - z) | + | \nabla (u - v) | | \nabla z | \right)^p
  \end{align*}
  and from this, using Minkowski's inequality, we are led to
  \begin{equation}
    \begin{split}
      \| & K ( u , y ) - K ( v , z ) \|_{L^p (\Omega)} \\
      & \leq (2 d)^{1/2}
      \left( \int_\Omega \left( | \nabla u | \: | \nabla (y - z) | + | \nabla (u - v) | \: | \nabla z | \right)^p \: dx \right)^{1/p} \\
      & \leq (2 d)^{1/2} \left( \left( \int_\Omega | \nabla u |^p | \nabla (y - z) |^p dx \right)^{1/p}
      + \left( \int_\Omega | \nabla (u - v) |^p | \nabla z |^p dx \right)^{1/p} \right) \: .
    \end{split}
  \end{equation}
  From this, H\"older's inequality gives
  \begin{equation}
    \begin{split}
      \| K ( u , y ) & - K ( v , z ) \|_{L^p (\Omega)} \\
      & \leq (2 d)^{1/2} \left( \| \nabla u \|_{L^q (\Omega)} \| \nabla (y - z) \|_{L^2 (\Omega)} + \| \nabla (u - v) \|_{L^2 (\Omega)} \| \nabla z \|_{L^q (\Omega)} \right)
    \end{split}
    \label{eq:K_bound_difference_general}
  \end{equation}
  with $q = 2 p /(2-p)$.
  
  (iii) The desired inequality (\ref{eq:K_bound_difference}) finally follows from inserting $u = u_\Omega$, $y = y_\Omega$, $v = u_{\Omega,h}$ and $z = y_{\Omega,h}$
  into (\ref{eq:K_bound_difference_general}) and using the fact that
  \[
    \| \nabla y_{\Omega,h} \|_{L^q (\Omega)} = \sup_{w_h} \frac{( \nabla y_{\Omega,h} , \nabla w_h )}{\| \nabla w_h \|_{L^{q^\star} (\Omega)}}
    = \sup_{w_h} \frac{( \nabla y_\Omega , \nabla w_h )}{\| \nabla w_h \|_{L^{q^\star} (\Omega)}} \leq \| \nabla y_\Omega \|_{L^q (\Omega)}
  \]
  holds due to Galerkin orthogonality.
\end{proof}

The other ingredients into $\eta_p (\Omega)$ and $\eta_{p,h}$ are the solutions $S \in \Sigma^p$ of (\ref{eq:KKT_4}) and $S_h \in \Sigma_h^p$ of (\ref{eq:KKT_h})
which are treated in the following lemma.

\begin{lemma}
  If $P_h^0$ and $P_h^{0,b}$ denote the closest point projections with respect to $L^p (\Omega)$ and $L^p (\partial \Omega)$, respectively, onto piecewise constants
  on $\cT_h$, then 
  \begin{equation}
    \begin{split}
      \| S - S_h \|_{L^p (\Omega)} & \leq \| K ( u_\Omega , y_\Omega ) - K ( u_{\Omega,h} , y_{\Omega,h} ) \|_{L^p (\Omega)} \\
      & \;\;\;\; + C \left( \| f \: \nabla y_\Omega - P_h^0 f \: \nabla y_{\Omega,h} \|_{L^p (\Omega)} \phantom{\int} \right.\\
      & \hspace{2cm} + \| j^\prime (u_\Omega) \: \nabla u_\Omega - P_h^0 j^\prime (u_{\Omega,h}) \: \nabla u_{\Omega,h} \|_{L^p (\Omega)} \\
      & \hspace{3cm}  \left. \phantom{\int} + \| j (u_\Omega) - P_h^{0,b} j (u_{\Omega,h}) \|_{W^{-1/p,p} (\partial \Omega)} \right)
    \end{split}
    \label{eq:S_bound_difference}
  \end{equation}
  holds with a constant $C$ (independently of $h$).
  \label{lemma-S_bound_difference}
\end{lemma}

\begin{proof}
  The system (\ref{eq:KKT_4}) implies that $S \in \Sigma^p$ is the closest point to $K ( u_\Omega , y_\Omega )$ out of the affine space
  \begin{equation}
    \Sigma_{\rm adm}^p = \{ T \in \Sigma^p : \div \: T = f \: \nabla y_\Omega - j^\prime (u_\Omega) \: \nabla u_\Omega \: , \:
    \left. T \cdot n \right|_{\partial \Omega} = \left. j (u_\Omega) \: n \right|_{\partial \Omega} \}
    \label{eq:affine_subspace_continuous}
  \end{equation}
  with respect to the $L^p (\Omega)$ norm. With the projections $P_h^0$ and $P_h^{0,b}$ defined above, we deduce from (\ref{eq:KKT_h}) that
  $S_h \in \Sigma_h^p$ is the closest point to $K ( u_{\Omega,h} , y_{\Omega,h} )$ out of
  \begin{equation}
    \begin{split}
      \Sigma_{{\rm adm},h}^p = \{ T_h \in \Sigma_h^p : \div \: T = P_h^0 f \: \nabla y_{\Omega,h} & - P_h^0 j^\prime (u_{\Omega,h}) \: \nabla j (u_{\Omega,h}) \: , \: \\
      & \left. T_h \cdot n \right|_{\partial \Omega} = \left. P_h^{0,b} j (u_{\Omega,h}) \: n \right|_{\partial \Omega} \}
    \end{split}
    \label{eq:affine_subspace_discrete}
  \end{equation}
  again with respect to the $L^p (\Omega)$ norm. Define $\tilde{S} \in \Sigma^p$ as the closest point to $K ( u_{\Omega,h} , y_{\Omega,h} )$ from
  $\Sigma_{\rm adm,h}^p$. The contraction property of the projection leads to
  \begin{equation}
    \| S - \tilde{S} \|_{L^p (\Omega)} \leq \| K ( u_\Omega , y_\Omega ) - K ( u_{\Omega,h} , y_{\Omega,h} ) \|_{L^p (\Omega)} \: .
    \label{eq:contraction}
  \end{equation}
  Both $\tilde{S}$ and $S_h$ are closest-point projections of $K ( u_{\Omega,h} , y_{\Omega,h} )$ but to the different affine spaces $\Sigma_{\rm adm}^p$ and
  $\Sigma_{{\rm adm},h}^p$, respectively. Therefore, we have
  \begin{equation}
    \begin{split}
      \| \tilde{S} - S_h \|_{L^p (\Omega)} & \leq \inf \{ \| \tilde{T} - T_h \|_{L^p (\Omega)} : \tilde{T} \in \Sigma_{\rm adm}^p \: , \: T_h \in \Sigma_{{\rm adm},h}^p \} \\
      & = \inf \{ \| T \|_{L^p (\Omega)} : \div \: T = f \: \nabla y_\Omega - P_h^0 f \: \nabla y_{\Omega,h} \\
      & \hspace{3cm} - ( j^\prime (u_\Omega) \: \nabla u_\Omega - P_h^0 j^\prime (u_{\Omega,h}) \: \nabla u_{\Omega,h} ) \: , \: \\
      & \hspace{3.5cm} T \cdot n |_{\partial \Omega} = ( j (u_\Omega) - P_h^{0,b} j (u_{\Omega,h}) ) n |_{\partial \Omega} \} \: .
    \end{split}
    \label{eq:affine_space_distance}
  \end{equation}
  The Helmholtz decomposition (for each row of $T$) in $L^p (\Omega;\R^{d \times d})$ (see \cite{FujMor:77}, \cite[Sect. III.1]{Gal:11}) states that the minimum in
  (\ref{eq:affine_space_distance}) is attained for $T = \nabla \vartheta$ with
  \begin{equation}
    \begin{split}
      ( \nabla \vartheta , \nabla \chi ) & = \langle j (u_\Omega) - P_h^{0,b} j (u_{\Omega,h}) , \chi \cdot n \rangle \\
      & \;\; - ( f \: \nabla y_\Omega - P_h^0 f \: \nabla y_{\Omega,h}
      - ( j^\prime (u_\Omega) \: \nabla u_\Omega,h - P_h^0 j^\prime (u_{\Omega,h}) \: \nabla u_{\Omega,h} ) , \chi )
    \end{split}
    \label{eq:Helmholtz_boundary_value_problem}
  \end{equation}
  for all $\chi \in W^{1,p^\ast} ( \Omega;\R^d )$. Inserting $\chi = \vartheta$ and using
  the Poincar\'e inequality
  \begin{equation}
    \| \chi \|_{L^{p^\ast} (\Omega)} \leq C_P \| \nabla \chi \|_{L^{p^\ast} (\Omega)} \mbox{ for all } \chi \in W^{1,p^\ast} (\Omega ; \R^d)
    \label{eq:Poincare}
  \end{equation}
  with $( \chi , e ) = 0$ for $e \in \R^d$ (cf. e.g. \cite[Ch. 3.3]{ErnGue:21a}) as well as the trace inequality
  \begin{equation}
    \begin{split}
      \| \chi \|_{W^{1/p,p^\ast} (\partial \Omega)} & \leq \tilde{C}_T \left( \| \chi \|_{L^{p^\ast} (\Omega)} + \| \nabla \chi \|_{L^{p^\ast} (\Omega)} \right) \\
      & \leq \tilde{C}_T ( 1 + C_P ) \| \nabla \chi \|_{L^{p^\ast} (\Omega)} =: C_T \| \nabla \chi \|_{L^{p^\ast} (\Omega)}
    \end{split}
    \label{eq:trace}
  \end{equation}
  (cf. \cite[Ch. 3.2]{ErnGue:21a}) gives
  \begin{equation}
    \begin{split}
      \| \nabla \vartheta \|_{L^p (\Omega)} & \leq C_T \| j (u_\Omega) - P_h^{0,b} j (u_{\Omega,h}) \|_{W^{-1/p,p} (\partial \Omega)} \\
      & \;\;\;\; + C_P \left( \| f \: \nabla y_\Omega - P_h^0 f \: \nabla y_{\Omega,h} \|_{L^p (\Omega)} \right. \\
      & \hspace{2cm} \left. + \| j^\prime (u_\Omega) \: \nabla u_\Omega - P_h^0 j^\prime (u_{\Omega,h}) \: \nabla u_{\Omega,h} \|_{L^p (\Omega)} \right) \: .
    \end{split}
    \label{eq:bound}
  \end{equation}
  Combining (\ref{eq:contraction}) and (\ref{eq:bound}) gives the desired inequality (\ref{eq:S_bound_difference}).
\end{proof}

Finally, the following result shows the closeness of the computed least mean functional to the exact one and thus confirms the significance of our computational results.

\begin{theorem}
  For $p \in ( 1 , 2 ]$, assume that $\Omega$ is such that the solutions of (\ref{eq:bvp}) and (\ref{eq:adjoint_bvp}) satisfy
  $\{ u_\Omega , y_\Omega \} \subset H^2 (\Omega) \cap W^{1,q} (\Omega)$ (again with $q = 2 p / (2 - p)$). Then, if $f \in H^1 (\Omega) \cap L^q (\Omega)$ and
  $j^\prime (u_\Omega) \in H^1 (\Omega) \cap L^q (\Omega)$, there is a constant $C$ such that
  \[
    \left| \eta_p (\Omega) - \eta_{p,h} (\Omega) \right| \leq C h \: .
  \]
  \label{theorem-eta_difference}
\end{theorem}

\begin{proof}
  Under our assumptions, Lemma \ref{lemma-K_bound_difference} gives
  \begin{equation}
    \begin{split}
      \| K ( u_\Omega , y_\Omega ) & - K ( u_{\Omega,h} , y_{\Omega,h} ) \|_{L^p (\Omega)} \\
      & \leq C h \left( \| \nabla u_\Omega \|_{L^q (\Omega)} | \nabla y_\Omega |_{H^1 (\Omega)}
      + \| \nabla y_\Omega \|_{L^q (\Omega)} | \nabla u_\Omega |_{H^1 (\Omega)} \right) \: ,
    \end{split}
    \label{eq:K_bound_difference_special}
  \end{equation}
  where we used the standard estimate for the Galerkin approximation (cf. e.g. \cite[Sect. 26.3]{ErnGue:21b}). Moreover, Lemma \ref{lemma-S_bound_difference}
  implies
  \begin{align*}
    \| S & - S_h \|_{L^p (\Omega)} \leq \| K ( u_\Omega , y_\Omega ) - K ( u_{\Omega,h} , y_{\Omega,h} ) \|_{L^p (\Omega)} \\
    & + C \left( \| f - P_h^0 f \|_{L^2 (\Omega)} \| \nabla y_\Omega \|_{L^q (\Omega)}
    + \| f \|_{L^q (\Omega)} \| \nabla y_\Omega - \nabla y_{\Omega_h} \|_{L^2 (\Omega)} \phantom{\int} \right.\\
    & + \| j^\prime (u_\Omega) - P_h^0 j^\prime (u_\Omega) \|_{L^2} \: \| \nabla u_\Omega \|_{L^q}
    + \| j^\prime (u_\Omega) - j^\prime (u_{\Omega,h}) \|_{L^2} \: \| \nabla u_\Omega \|_{L^q} \\
    & \phantom{\int} + \left. \| j^\prime (u_\Omega) \|_{L^q} \| \nabla u_\Omega - \nabla u_{\Omega_h} \|_{L^2}
    + \| j (u_\Omega) - P_h^{0,b} j (u_{\Omega,h}) \|_{W^{-1/p,p} (\partial \Omega)} \right) \\
    \leq & \: \| K ( u_\Omega , y_\Omega ) - K ( u_{\Omega,h} , y_{\Omega,h} ) \|_{L^p (\Omega)} \\
    & + C h \left( | f |_{H^1 (\Omega)} \| \nabla y_\Omega \|_{L^q (\Omega)} + \| f \|_{L^q (\Omega)} | \nabla y_\Omega |_{H^1 (\Omega)} 
    + | j^\prime (u_\Omega) |_{H^1 (\Omega)} \: \| \nabla u_\Omega \|_{L^q (\Omega)} \phantom{\int} \right. \\
    & \hspace{1cm} \left. \phantom{\int} 
    + \| j^\prime (u_\Omega) \|_{L^q (\Omega)} | \nabla u_\Omega |_{H^1 (\Omega)} + \| j (u_\Omega) \|_{W^{1-1/p,p} (\partial \Omega)} \right) \: .
  \end{align*}
  The statement of Theorem \ref{theorem-eta_difference} follows now from
  \[
    | \eta_p (\Omega) - \eta_{p,h} (\Omega) | \leq \| S - S_h \|_{L^p (\Omega)} + \| K ( u_\Omega , y_\Omega ) - K ( u_{\Omega,h} , y_{\Omega,h} ) \|_{L^p (\Omega)} \: .
  \]
\end{proof}

\begin{remark}
  The assumptions in Theorem \ref{theorem-eta_difference} are fulfilled in our examples since $\Omega$ is either convex or has a smooth boundary.
\end{remark}

\bibliography{../../../biblio/msc_articles,../../../biblio/msc_books}
\bibliographystyle{siamplain}

\end{document}